\documentclass[draft]{amsart}

\usepackage[british]{babel}
\usepackage{amssymb}
\usepackage{cite}
\usepackage[dvipsnames]{xcolor} 

\newcommand{\eql}{\kern-1ex &=& \kern-1ex}
\newcommand{\Z}{\mathbb{Z}}
\newcommand{\R}{\mathbb{R}}
\newcommand{\C}{\mathbb{C}}
\newcommand{\wj}{\mathrm{wj}}
\newcommand{\Ker}{\mathop\mathrm{Ker}\nolimits}
\newcommand{\coKer}{\mathop\mathrm{coKer}\nolimits}

\newcommand{\ind}{\mathop\mathrm{ind}\nolimits}

\newcommand{\sign}{\mathop\mathrm{sign}\nolimits}

\newcommand{\Iso}{\mathrm{Iso}}
\newcommand{\GL}{\mathrm{GL}}
\newcommand{\A}{\mathcal{A}}
\newcommand{\M}{\mathcal{M}}
\newcommand{\N}{\mathcal{N}}
\newcommand{\I}{\mathcal{I}}
\newcommand{\X}{\mathcal{X}}
\newcommand{\K}{\mathcal{K}}
\newcommand{\pb}{\kern -1ex}

\newcommand{\e}{\varepsilon}
\newcommand{\psideg}{\Psi\hbox{-}\mathrm{deg}}
\newcommand{\psidegree}{\Psi\hbox{-}{degree}}
\newcommand{\s}{\sigma}
\newcommand{\per}{\!\times\!}

\renewcommand{\L}{\mathcal{L}}
\renewcommand{\S}{\mathrm{S}}

\renewcommand{\Im}{\mathop\mathrm{Img}\nolimits}

\renewcommand{\l}{\lambda}

\renewcommand{\j}{\mathrm{j}}
\renewcommand{\a}{\alpha}
\renewcommand{\b}{\beta}
\newcommand{\g}{\gamma}
\newcommand{\G}{\Gamma}
\renewcommand{\o}{\omega}
\renewcommand{\O}{\Omega}
\renewcommand{\t}{\theta}

\newcommand{\F}{\mathcal{F}}

\theoremstyle{plain}
\newtheorem{theorem}{Theorem}[section]
\newtheorem{corollary}[theorem]{Corollary}
\newtheorem{lemma}[theorem]{Lemma}
\newtheorem{proposition}[theorem]{Proposition}

\newtheorem{remark}[theorem]{Remark}
\theoremstyle{definition}
\newtheorem{notation}[theorem]{Notation}
\newtheorem{definition}[theorem]{Definition}
\newtheorem{example}[theorem]{Example}

\numberwithin{equation}{section}

\begin{document}
\title[A degree associated to linear eigenvalue problems in Hilbert spaces]{A degree associated to linear eigenvalue problems in Hilbert spaces and applications to nonlinear spectral theory}

\author[P.\ Benevieri]{Pierluigi Benevieri}
\author[A.\ Calamai]{Alessandro Calamai}
\author[M.\ Furi]{Massimo Furi}
\author[M.P.\ Pera]{Maria Patrizia Pera}

\thanks{The first, second and fourth authors are members of the Gruppo Nazionale per l'Analisi Matematica, la Probabilit\`a e le loro Applicazioni (GNAMPA) of the Istituto Nazionale di Alta Matematica (INdAM)}
\thanks{A.\ Calamai is partially supported by GNAMPA\ - INdAM (Italy)}

\date{\today}

\address{Pierluigi Benevieri -
Instituto de Matem\'atica e Estat\'istica,
Universidade de S\~ao Paulo,
Rua do Mat\~ao 1010,
S\~ao Paulo - SP - Brasil - CEP 05508-090 -
 {\it E-mail address: \tt
pluigi@ime.usp.br}}
\address{Alessandro Calamai -
Dipartimento di Ingegneria Civile, Edile e Architettura,
Universit\`a Politecnica delle Marche,
Via Brecce Bianche,
I-60131 Ancona, Italy -
 {\it E-mail address: \tt
calamai@dipmat.univpm.it}}
\address{Massimo Furi - Dipartimento di Matematica e Informatica ``Ulisse Dini'',
Uni\-ver\-sit\`a degli Studi di Firenze,
Via S.\ Marta 3, I-50139 Florence, Italy -
{\it E-mail address: \tt
massimo.furi@unifi.it}}
\address{Maria Patrizia Pera - Dipartimento di Matematica e Informatica ``Ulisse Dini'',
Universit\`a degli Studi di Firenze,
Via S.\ Marta 3, I-50139 Florence, Italy -
{\it E-mail address: \tt
mpatrizia.pera@unifi.it}}

\begin{abstract}
We extend to the infinite dimensional context the link between two completely different topics recently highlighted by the authors: the classical eigenvalue problem for real square matrices and the Brouwer degree for maps between oriented finite dimensional real manifolds.
Thanks to this extension, we solve a conjecture regarding global continuation in nonlinear spectral theory that we have formulated in a recent article.
Our result (the ex conjecture) is applied to prove a Rabinowitz type global continuation property of the solutions to a perturbed motion equation containing an air resistance frictional force.
\end{abstract}



\dedicatory{Dedicated to the memory of our friend and\\ outstanding mathematician Russell Johnson}

\maketitle

\section{Introduction}
\label{Introduction}

Given a linear operator $L \colon \R^k \to \R^k$, consider the Classical Eigenvalue Problem
\begin{equation}
\label{intro - classical eigenvalue problem}
\begin{cases}
\;L x = \l x,\\[.3ex]
\; x \in \S,
\end{cases}
\end{equation}
where $\S$ is the unit sphere of $\R^k$ and $\l \in \R$.
The solutions of \eqref{intro - classical eigenvalue problem} are pairs $(\l,x) \in \R\per\S$, hereafter called \emph{eigenpoints}, where $\l$ is a real eigenvalue of $L$ and $x$ is one of the corresponding unit eigenvectors.
Since the eigenpoints of \eqref{intro - classical eigenvalue problem} are the zeros of the $C^\infty$-map
\[
\Psi \colon \R\per\S \to \R^k, \quad (\l,x) \mapsto Lx-\l x,
\]
in \cite{BeCaFuPe-s5} we have shown that there is a link between the above purely algebraic problem and the Brouwer degree, $\deg(\Psi,U,0)$, of $\Psi$ with target $0 \in \R^k$ on convenient open subsets $U$ of the cylinder $\R\per\S$, which is a smooth $k$-dimensional real manifold with a natural orientation.

Roughly speaking, in \cite{BeCaFuPe-s5} we have shown that
\begin{itemize}
\item
\emph{if $\l_* \in \R$ is a simple eigenvalue of $L$, and $x_*$ and $-x_*$ are the two corresponding unit eigenvectors, then the ``twin'' eigenpoints $p_* = (\l_*,x_*)$ and $\bar p_* = (\l_*,-x_*)$ are isolated zeros of $\Psi$ and give the same contribution to the Brouwer degree, which is either $1$ or $-1$, depending on the sign jump at $\l_*$ of the (real) characteristic polynomial of $L$.}
\end{itemize}

Still roughly speaking, here we extend this fact to the infinite dimensional case by considering a problem of the type
\begin{equation}
\label{intro - eigenvalue problem in Hilbert spaces}
\begin{cases}
\;L x = \l Cx,\\[.3ex]
\; x \in \S,
\end{cases}
\end{equation}
in which $\l \in \R$, $L$ and $C$ are bounded linear operators acting between two real Hilbert spaces $G$ and $H$, $C$ is compact, $L-\l C$ is invertible for some $\l \in \R$, and $\S$ is the unit sphere of the source space $G$.

Even in the special case $G = H$ and $L$ the identity, Problem \eqref{intro - eigenvalue problem in Hilbert spaces} cannot be treated using the degree of Leray and Schauder which, unlike that of Brouwer, does not hold in the context of smooth manifolds.
Therefore, we apply the degree introduced in \cite{BeFu1} for oriented $C^1$ Fredholm maps of index zero between real differentiable Banach manifolds, which extends the Brouwer degree for maps between oriented finite dimensional smooth manifolds, as well as the Leray-Schauder degree for $C^1$ compact vector fields in Banach spaces.
To apply this degree we need the unit sphere $\S$ to be a smooth manifold.
This is the reason of our restriction to Hilbert spaces instead of the more general Banach environment.

Here the degree regards the smooth map
\[
\Psi \colon \R\per\S \to H, \quad (\l,x) \mapsto Lx-\l Cx,
\]
acting between the $1$-codimensional submanifold $\R\per\S$ of the Hilbert space $\R\per G$ and the target space $H$, whose zeros are called the \emph{eigenpoints} of \eqref{intro - eigenvalue problem in Hilbert spaces}.
The result obtained here, Theorem \ref{psi-degree simple eigenpoint} below, extends, to the infinite dimensional case, the one in \cite{BeCaFuPe-s5} concerning the Classical Eigenvalue Problem \eqref{intro - classical eigenvalue problem}, provided that one calls $\l_* \in \R$ a \emph{simple eigenvalue} of \eqref{intro - eigenvalue problem in Hilbert spaces} if there exists $x_* \in \S$ such that $\Ker(L-\l_*C)= \R x_*$ and $H = \Im(L-\l_*C) \oplus \R Cx_*$.
In fact, we obtain that
\begin{itemize}
\item
\emph{if $\l_*$ is a simple eigenvalue of \eqref{intro - eigenvalue problem in Hilbert spaces} and $x_*$ and $-x_*$ are the two corresponding unit eigenvectors, then the ``twin'' eigenpoints $p_* = (\l_*,x_*)$ and $\bar p_* = (\l_*,-x_*)$ are isolated zeros of $\Psi$ and give the same contribution to the degree, which is either $1$ or $-1$, depending on the orientation of~$\Psi$.}
\end{itemize}

\medskip
As in \cite{BeCaFuPe-s5}, this crucial result regarding the ``fair contribution to the degree of the twin eigenpoints'' is applied to the study of the behaviour of the solution triples $(s,\l,x)$ of the following perturbation of \eqref{intro - eigenvalue problem in Hilbert spaces}:
\begin{equation}
\label{intro - perturbed eigenvalue problem in Hilbert spaces}
\begin{cases}
\;L x + s N(x) = \l Cx,\\[.3ex]
\; x \in \S,
\end{cases}
\end{equation}
in which $N\colon \S \to H$ is a compact $C^1$-map.
Precisely, if we denote by $\Sigma$ the subset of $\R\per\R\per\S$ of the solutions $(s,\l,x)$ of \eqref{intro - perturbed eigenvalue problem in Hilbert spaces} and we call \emph{trivial} those having $s=0$, our main result, Theorem \ref{ex conjecture} below, yields the following Rabinowitz type global continuation result that was conjectured in~\cite{BeCaFuPe-s4}.
\begin{itemize}
\item
\emph{If $q_* = (0,\l_*,x_*)$ is a trivial solution of \eqref{intro - perturbed eigenvalue problem in Hilbert spaces} corresponding to a simple eigenvalue $\l_*$ of the unperturbed problem \eqref{intro - eigenvalue problem in Hilbert spaces}, then the connected component of $\Sigma$ containing $q_*$ is either unbounded or encounters a trivial solution $q^* = (0,\l^*,x^*)$ with $\l^* \not=\l_*$.}
\end{itemize}

The result obtained in \cite{BeCaFuPe-s4} differs from this one in the final assertion and the additional assumption that $G$ and $H$ are separable.
In fact, in \cite{BeCaFuPe-s4}, without any degree theory and with arguments of differential topology, we proved that
\begin{itemize}
\item
\emph{the connected component of $\Sigma$ containing $q_* = (0,\l^*,x^*)$ is either unbounded or encounters a trivial solution $q^* = (0,\l_*,x_*)$ different from $q_*$.}
\end{itemize}
Notice that this assertion does not exclude the case $\l^* = \l_*$ and, consequently, $x^*=-x_*$, since $\l_*$ is simple.

\medskip
We believe that our global continuation result, Theorem \ref{ex conjecture}, could be extended to the more general situation in which $\Ker(L-\l_*C)$ is odd dimensional and
\[
H = \Im(L-\l_*C) \oplus C(\Ker(L-\l_*C)).
\]
Our belief is supported by the fact that in \cite{BeCaFuPe-s2}, under these assumptions, we have shown that
\begin{itemize}
\item
\emph{the projection of $\Sigma$ into the $s\l$-plane has a connected set which contains $(0,\l_*)$ and is either unbounded or includes a ``trivial eigenpair'' $(0,\l^*)$ different from $(0,\l_*)$.}
\end{itemize}
Moreover, our belief is also based on the fact that in \cite{BeCaFuPe-s5}, by means of the Brouwer degree for maps between oriented smooth manifolds, we obtained that
\begin{itemize}
\item
\emph{if $G=H=\R^k$, $C$ is the identity, and $x_*$ is a unit eigenvector of $L$ corresponding to an eigenvalue $\l_*$ with odd algebraic multiplicity, then the connected component of $\Sigma$ containing $(0,\l_*,x_*)$ is either unbounded or includes a trivial solution $(0, \l^*, x^*)$ with $\l^* \not= \l_*$.}
\end{itemize}
So far, our effort to obtain the supposed extension of Theorem \ref{ex conjecture} has proved unsuccessful.

\medskip
Theorem \ref{ex conjecture} falls into the subject of \emph{nonlinear spectral theory},
which finds applications to differential equations
(see e.g.\ \cite{ADV,Chi2018} and references therein).

Here we mention the work of R.\ Chiappinelli \cite{Chi}, which inspired some of our recent articles.
In \cite{Chi} a sort of ``local persistence property'' for a perturbed eigenvalue problem similar to \eqref{intro - perturbed eigenvalue problem in Hilbert spaces} was proved.
Precisely, under the assumptions
\begin{itemize}
\item
$L\colon G \to G$ is a \emph{self-adjoint} operator,
\item
$C = I$ is the identity of $G$,
\item
$N \colon \S \to G$ is Lipschitz continuous,
\item
$\l_* \in \R$ is an isolated simple eigenvalue of $L$,
\item
$x_* \in \S$ is an eigenvector corresponding to $\l_*$,
\end{itemize}
it was shown that
\begin{itemize}
\item
\emph{in a neighborhood $V$ of $0 \in \R$ a Lipschitz continuous function $\e \mapsto (\l_\e,x_\e) \in \R\per\S$ is defined with the properties that $(\l_0,x_0) = (\l_*,x_*)$ and that $Lx_\e + \e N(x_\e) = \l_\e x_\e$ for any $\e \in V$.}
\end{itemize}

When $G$ is infinite dimensional, the hypotheses of our global continuation result seem incompatible with the assumptions of Chiappinelli, due to the fact that the identity is not a compact operator.
However, Theorem \ref{ex conjecture} does apply to the equation
\[
Lx+\e N(x) = \l x,
\]
provided that $N$ is $C^1$ and compact, and $L$ is of the type $\l_* I + C$, with $\l_* \in \R$ and $C$ compact.
In fact, putting $\e= -\s/\mu$ and $\l=\l_*+1/\mu$, the above equation becomes $x + \s N(x) = \mu Cx$, which is as in our problem \eqref{intro - perturbed eigenvalue problem in Hilbert spaces} with $L=I$.

\medskip
For results regarding the local as well as global  persistence property when the eigenvalue $\l_*$ is not necessarily simple we mention \cite{BeCaFuPe-s1,BeCaFuPe-s2,BeCaFuPe-s3,BeCaFuPe-s4,BeCaFuPe-s5, Chi2017, ChFuPe1, ChFuPe2, ChFuPe3, ChFuPe4,ChFuPe5}.

\medskip
The last section of the paper contains three examples illustrating our main result, as well as an application to the study of the solutions $(s,\l,x)$ of the boundary value problem
\begin{equation}
\label{intro - air resistance}
\left\{
\begin{array}{lccc}
x''(t) + s g(x'(t)) + \l x(t) = 0,\\
x(0) = 0 = x(\pi),\; x \in \S,
\end{array}\right.
\end{equation}
in which $\S$ is the unit sphere of the Hilbert space $H^2(0,\pi)$, and $g\colon \R \to \R$ is an odd increasing $C^1$-function (such as the air resistance force $g(v) = v|v|$).
From our result, with the help of the well-known notion of winding number of a self-map of the circle $S^1$, we deduce that, given any trivial solution $q_*$ of \eqref{intro - air resistance}, the connected component of $\Sigma$ containing $q_*$ is unbounded and does not encounter other trivial solutions.

For pioneering articles regarding the use of the winding number in order to study the behavior of solutions to ordinary differential equations we mention \cite{CaMaZa,FuPe1,FuPe2}.

\section{Notation and preliminaries}
\label{Preliminaries}

We introduce some notation, preliminaries, and known or unknown concepts that we will need in subsequent sections.
In particular, we will outline the main notions related to the topological degree for oriented $C^1$ Fredholm maps of index zero between real differentiable Banach manifolds introduced in \cite{BeFu1} (see also \cite{BeFu2,BeFu5} for additional details).
Actually, some notions and results are new: we consider them necessary for a better understanding of the topics in Sections \ref{Results 1} and \ref{Results 2}, the proof of Theorem \ref{psi-degree simple eigenpoint} in particular.

\subsection{Algebraic preliminaries}
\label{Algebraic preliminaries}
Let, hereafter, $E$ and $F$ be two real vector spaces.
By $\L(E,F)$ we shall denote the vector space of the linear operators from $E$ into $F$. The same notation will be used if, in addition, $E$ and $F$ are normed. In this case, however, we will tacitly assume that all the operators of $\L(E,F)$ are bounded, and that this space is endowed with the usual operator norm. If $E=F$, we will write $\L(E)$ instead of $\L(E,E)$.
By $\mathrm{Iso}(E,F)$ we shall mean the subset of $\L(E,F)$ of the invertible operators, and we will write $\GL(E)$ instead of $\Iso(E,E)$.
The subspace of $\L(E,F)$ of the operators having finite dimensional image will be denoted by $\F(E,F)$, or simply by $\F(E)$ when $E=F$.

\medskip
Let $I \in \L(E)$ indicate the identity of $E$.
If $T \in \L(E)$ has the property that $I-T \in \F(E)$, we shall say that $T$ is an \emph{admissible operator (for the determinant)}.
The symbol $\A(E)$ will stand for the affine subspace of $\L(E)$ of the admissible operators.

It is known (see \cite{Ka}) that the determinant of an operator $T \in \A(E)$ is well defined as follows: $\det T := \det T|_{\hat E}$, where $T|_{\hat E}$ is the restriction (as domain and as codomain) to any finite dimensional subspace $\hat E$ of $E$ containing $\Im(I-T)$, with the understanding that $\det T|_{\hat E} = 1$ if $\hat E = \{0\}$.

As one can easily check, the function $\det\colon \A(E) \to \R$ inherits most of the properties of the classical determinant.
Some of them are stated in the following

\begin{remark}
\label{properties of determinant}
Let $T, T_1, T_2 \in \A(E)$. Then
\begin{itemize}
\item
$\det T \not= 0$ if and only if $T$ is invertible;
\item
$R\in \Iso(E,F)$ implies $RTR^{-1} \in \A(F)$ and $\det(RTR^{-1}) = \det T$;
\item
$T_2T_1 \in \A(E)$ and $\det(T_2T_1) = \det(T_2)\det(T_1)$.
\end{itemize}
\end{remark}

See, for example, \cite{BeFuPeSp07} for a discussion about other properties of the determinant.

\medskip
We will need the following remark, whose easy proof is left to the reader:

\begin{remark}
\label{determinant}
Let $T \in \L(E)$ and let $E = E_1 \oplus E_2$ with $\dim E_2 < +\infty$.
Assume that, with respect to the above decomposition, $T$ can be represented in a block matrix form
\[
T =
\left(
\begin{array}{cc}
I_{11} & T_{12}\\[2ex]
0 & T_{22}
\end{array}
\right),
\]
where $I_{11}$ is the identity of $E_1$.
Then $T \in \A(E)$ and $\det T = \det T_{22}$.
\end{remark}

Recall that an operator $T \in \L(E,F)$ is said to be \emph{(algebraic) Fredholm} if its kernel, $\Ker T$, and its cokernel, $\coKer T = F/T(E)$, are both finite dimensional.

The \emph{index} of a Fredholm operator $T$ is the integer
\[
\ind T = \dim(\Ker T) - \dim(\coKer T).
\]
In particular, any invertible linear operator is Fredholm of index zero.
Observe also that, if $T \in \L(\R^k,\R^s)$, then $\ind T = k-s$.

The subset of $\L(E,F)$ of the Fredholm operators will be denoted by $\Phi(E,F)$; while $\Phi_n(E,F)$ will stand for the set $\{T \in \Phi(E,F): \ind T = n\}$.
Obviously, $\Phi(E)$ and $\Phi_n(E)$ designate, respectively, $\Phi(E,E)$ and $\Phi_n(E,E)$.

\medskip
One can easily check that $\A(E)$ is a subset of $\Phi_0(E)$.
This is also a consequence of a well known property regarding Fredholm operators.
Namely,
\begin{itemize}
\item
[(1)] \emph{if $T \in \Phi_n(E,F)$ and $K \in \F(E,F)$, then $T+K \in \Phi_n(E,F)$.}
\end{itemize}

Another fundamental property states that
\begin{itemize}
\item
[(2)] \emph{the composition of Fredholm operators is Fredholm and its index is the sum of the indices of all the composite operators.}
\end{itemize}

An useful consequence of property (2) is the following:
\begin{itemize}
\item
If $T \in \Phi_n(E,F)$ and $k \in \mathbb N$, then the restriction of $T$ to a $k$-codimensional subspace of $E$ is Fredholm of index $n-k$.
\end{itemize}

\medskip
Let $T \in \Phi_0(E,F)$.
In \cite{BeFu1}, an operator $K \in \F(E,F)$ was called a \emph{corrector of $T$} if $T+K$ is invertible.
Since, during a conference, someone has critically observed that it is not necessary to correct an invertible operator, hereafter we will use the more appropriate word \emph{companion} instead of \emph{corrector}.

Notice that any $T \in \Iso(E,F)$ has a \emph{natural companion}: the trivial element of $\L(E,F)$.
This fact was crucial in \cite{BeFu1} for the construction of the degree theory that we will apply here.

Given $T \in \Phi_0(E,F)$, let us denote by $\mathcal C(T)$ the subset of $\F(E,F)$ of all the companions of $T$.
As one can easily check, this set is nonempty.
Moreover, $\mathcal C(T)$ admits a partition in two equivalence classes according to the following

\begin{definition}[Equivalence relation]
\label{equivalence relation}
Two companions $K_1$ and $K_2$ of an operator $T \in \Phi_0(E,F)$ are \emph{equivalent} (more precisely, \emph{$T$-equivalent)} if the admissible operator $(T+K_2)^{-1}(T+K_1)$ has positive determinant.
\end{definition}

Given two companions $K_1$ and $K_2$ of $T \in \Phi_0(E,F)$, the admissible automorphism $(T+K_2)^{-1}(T+K_1)$ is not the unique one that can be used to check whether or not $K_1$ and $K_2$ are equivalent.
In fact, one has the following

\begin{remark}
\label{four determinants}
Let $T \in \Phi_0(E,F)$ and $K_1, K_2 \in \mathcal C(T)$.
Then, the determinants of the invertible operators
\[
(T+K_2)^{-1}(T+K_1),\; (T+K_1)(T+K_2)^{-1},\; (T+K_1)^{-1}(T+K_2),\; (T+K_2)(T+K_1)^{-1}
\]
have the same sign.
In fact, from the second property of Remark \ref{properties of determinant} one gets that the first two operators have the same determinant,
while the third property implies the statement regarding the last two operators, being the inverses of the first two.
\end{remark}

Thanks to the above equivalence relation, the following definition was introduced in \cite{BeFu1}.

\begin{definition}[Orientation]
\label{orientation of T}
An \emph{orientation} of $T \in \Phi_0(E,F)$ is one of the two equivalence classes of $\mathcal C(T)$, denoted by $\mathcal C_+(T)$ and called the class of \emph{positive companions} of the \emph{oriented operator} $T$.
The set $\mathcal C_-(T) = \mathcal C(T) \setminus C_+(T)$ of the \emph{negative companions} is the \emph{opposite orientation of $T$}.
\end{definition}

Some further definitions are in order.

\begin{definition}[Natural orientation]
\label{natural orientation}
Any $T \in \Iso(E,F)$ admits the \emph{natural orientation}: the one given by considering the trivial operator of $\L(E,F)$ as a positive companion.
\end{definition}

\begin{definition}[Canonical orientation]
\label{canonical orientation}
Any admissible operator $T \in \A(E)$ admits the \emph{canonical orientation}: the one given by choosing as a positive companion any $K \in \F(E)$ such that $\det(T+K) > 0$.
In particular, this applies for any $T \in \L(E)$, with $\dim E < \infty$.
\end{definition}

\begin{definition}[Associated orientation]
\label{associated orientation}
Let $E$ and $F$ have the same finite dimension. Assume that they are oriented up to an inversion of both the orientations or, equivalently, assume that $E\per F$ has an orientation, say $\mathcal O$.
Then any $T \in \L(E,F)$ admits the orientation \emph{associated with $\mathcal O$}: the one given by choosing as a positive companion any $K \in \F(E,F)$ such that $T+K$ is orientation preserving.
\end{definition}

\begin{definition}[Oriented composition]
\label{oriented composition}
The \emph{oriented composition} of two oriented operators, $T_1 \in \Phi_0(E_1,E_2)$ and $T_2 \in \Phi_0(E_2,E_3)$, is the operator $T_2T_1$ with the orientation given by considering $K = (T_2+K_2)(T_1+K_1)-T_2T_1$ as a positive companion, where $K_1$ and $K_2$ are positive companions of $T_1$ and $T_2$, respectively.
\end{definition}

Observe that the oriented composition is associative.
Indeed, if $T_1 \in \Phi_0(E_1,E_2)$, $T_2 \in \Phi_0(E_2,E_3)$ and $T_3 \in \Phi_0(E_3,E_4)$, and $K_1$, $K_2$ and $K_3$ are, respectively, companions of $T_1$, $T_2$ and $T_3$, one has
\begin{align*}
\big((T_3+K_3)(T_2+K_2)\big)(T_1+K_1)-\big(T_3T_2\big)T_1\;\\
=(T_3+K_3)\big((T_2+K_2)(T_1+K_1)\big)-T_3\big(T_2T_1\big).
\end{align*}

The following result implies an important property of the oriented composition (see Corollary \ref{sign of composition} below). Moreover, it shows that Definition \ref{oriented composition} is well-posed.

\begin{lemma}
\label{inversion in an oriented composition}
Given $T_1 \in \Phi_0(E_1,E_2)$, $T_2 \in \Phi_0(E_2,E_3)$, $K_1, K_1' \in \mathcal C(T_1)$ and $K_2, K_2' \in \mathcal C(T_2)$, consider the following companions of $T_2T_1$:
\[
K = (T_2+K_2)(T_1+K_1)-T_2T_1 \quad \text{and} \quad K' = (T_2+K_2')(T_1+K_1')-T_2T_1.
\]
Then $K$ is equivalent to $K'$ if and only if $K_1$ and $K_2$ are both equivalent or both not equivalent to $K_1'$ and $K_2'$, respectively.
\end{lemma}
\begin{proof}
According to Definition \ref{orientation of T}, we need to compute the sign of
\[
 \det\big((T_2T_1+K)^{-1}(T_2T_1+K')\big).
\]
We have
\begin{align*}
(T_2T_1+K)^{-1}(T_2T_1+K')
= \big((T_2+K_2)(T_1+K_1)\big)^{-1}\big((T_2+K_2')(T_1+K_1')\big).
\end{align*}
Thus, because of the second property in Remark \ref{properties of determinant}, we obtain
\begin{align*}
&\;\det\big((T_2T_1+K)^{-1}(T_2T_1+K')\big)\\
= &\; \det\Big((T_1+K_1')\big((T_2+K_2)(T_1+K_1)\big)^{-1}\big((T_2+K_2')(T_1+K_1')\big)(T_1+K_1')^{-1}\Big)\\
= &\; \det\big((T_1+K_1')(T_1+K_1)^{-1}(T_2+K_2)^{-1}(T_2+K_2')\big).
\end{align*}
Therefore, applying the third property of the same remark, we get
\begin{align*}
&\;\det\big((T_2T_1+K)^{-1}(T_2T_1+K')\big)\\
= &\; \det\big((T_1+K_1')(T_1+K_1)^{-1})\big)\det\big((T_2+K_2)^{-1}(T_2+K_2')\big),
\end{align*}
and the assertion follows.
\end{proof}

\begin{definition}[Sign of an oriented operator]
\label{sign}
Let $T \in \Phi_0(E,F)$ be an oriented operator.
Its \emph{sign} is the integer
\[
\sign T =
\left\{
\begin{array}{rl}
+1 & \mbox{if } T \mbox{ is invertible and naturally oriented,}\\
-1 & \mbox{if } T \mbox{ is invertible and not naturally oriented,}\\
 0 & \mbox{if } T \mbox{ is not invertible.}
\end{array}
\right.
\]
\end{definition}

As a straightforward consequence of Remark \ref{four determinants}, and taking into account Definitions \ref{equivalence relation}, \ref{orientation of T}, \ref{natural orientation}, \ref{sign},
we get the following

\begin{remark}
\label{sign test}
Let $T \in \Iso(E,F)$ be oriented.
Then,
\begin{align*}
\sign T &= \sign \det\big((T+K)^{-1}T\big) = \sign \det\big(T(T+K)^{-1}\big)\\
&= \sign \det\big(T^{-1}(T+K)\big) = \sign \det\big((T+K)T^{-1}\big),
\end{align*}
where $K$ is a positive companion of $T$.
\end{remark}

\medskip
Lemma \ref{inversion in an oriented composition} shows that, in the oriented composition, the inversion of the orientation of one (and only one) of the operators yields the inversion of the orientation of the composition.
Hence, one gets the following

\begin{corollary}
\label{sign of composition}
Let $T_1 \in \Phi_0(E_1,E_2)$ and $T_2 \in \Phi_0(E_2,E_3)$ be oriented.
Then, $\sign(T_2T_1) = \sign T_2 \sign T_1$, where $T_2T_1$ is the oriented composition of $T_1$ and $T_2$.
\end{corollary}
\begin{proof}
If one of the two operators is not invertible, then the assertion is obvious.
Assume, therefore, that $T_1$ and $T_2$ are isomorphisms.
If both the operators are naturally oriented, then the assertion follows from the definition of oriented composition.
The other cases are a consequence of Lemma \ref{inversion in an oriented composition}.
\end{proof}

Given $R_1 \in \Iso(E_1,F_1)$ and $R_2 \in \Iso(E_2,F_2)$, observe that the function
\[
\Lambda\colon \L(E_1,E_2) \to \L(F_1,F_2), \quad T \mapsto R_2TR_1^{-1}
\]
is a linear isomorphism (whose inverse is given by $\widetilde T \mapsto R_2^{-1}\widetilde TR_1$).
One can see that $\Lambda$ establishes a correspondence between certain pairs of subsets of $\L(E_1,E_2)$ and $\L(F_1,F_2)$.
For example, $\Iso(E_1,E_2)$ and $\Iso(F_1,F_2)$, $\F(E_1,E_2)$ and $\F(F_1,F_2)$, $\Phi_0(E_1,E_2)$ and $\Phi_0(F_1,F_2)$.
Moreover if, in particular, $T \in \Phi_0(E_1,E_2)$, then $\Lambda$ sends the set $\mathcal C(T)$ onto the set $\mathcal C(\Lambda(T))$, and if $K_1, K_2 \in \mathcal C(T)$ are equivalent (according to Definition \ref{equivalence relation}), so are $\Lambda(K_1), \Lambda(K_2) \in \mathcal C(\Lambda(T))$.

\medskip
Since the oriented composition is associative, this notion can be extended to the composition of three (or more) oriented operators.

\smallskip
\subsection{Topological preliminaries}
\label{Topological preliminaries}
Let, hereafter, $X$ denote a metric space.
We recall that a subset $A$ of $X$ is \emph{locally compact} if any point of $A$ admits a neighborhood, in $A$, which is compact.
Therefore, any compact subset of $X$ is locally compact, as is any relatively open subset of a locally compact set. However, the union of two locally compact subsets of $X$ may not be locally compact.

We recall also that a continuous map between metric spaces is said to be \emph{proper} if the inverse image of any compact set is a compact set, while it is called \emph{locally proper} if its restriction to a convenient closed neighborhood of any point of its domain is proper.
Thus, level sets of locally proper maps are locally compact.

One can check that proper maps are \emph{closed}, in the sense that the image of any closed set is a closed set.

\begin{notation}
\label{slice}
Let $D$ be a subset of the product $X\per Y$ of two metric spaces.
Given $x \in X$, we call \emph{$x$-slice of $D$} the set
$D_x = \{y \in Y: (x,y) \in D\}$.
\end{notation}

Assume, from now on, that the vector spaces $E$ and $F$ are actually Banach.
In this framework, any Fredholm operator is assumed to be bounded.
Therefore, in addition to the algebraic properties (1) and (2) stated in Subsection \ref{Algebraic preliminaries}, one has the following topological ones (see e.g.\ \cite{TaLa}):
\begin{itemize}
\item
[(3)] \emph{if $T \in \Phi(E,F)$, then $\Im T$ is closed in $F$;}
\item
[(4)] \emph{if $T \in \Phi(E,F)$, then $T$ is proper on any bounded closed subsets of $E$;}
\item
[(5)] \emph{for any $n \in \Z$, the set $\Phi_n(E,F)$ is open in $\L(E,F)$;}
\item
[(6)] \emph{if $T \in \Phi_n(E,F)$ and $K \in \L(E,F)$ is compact, then $T+K \in \Phi_n(E,F)$.}
\end{itemize}

\medskip
Let us now sketch the construction and summarize the main properties of the degree introduced in \cite{BeFu1}.

The basic fact is that, in the context of Banach spaces, the orientation of an operator $T_* \in \Phi_0(E,F)$ induces an orientation to the operators in a neighborhood of $T_*$.
Indeed, due to the fact that $\Iso(E,F)$ and $\Phi_0(E,F)$ are open in $\L(E,F)$, any companion of $T_*$ remains a companion of all $T$ sufficiently close to $T_*$.

\medskip
Therefore, the following definition makes sense.

\begin{definition}
\label{orientation of Gamma}
Let $\G\colon X \to \Phi_0(E,F)$ be a continuous map defined on a metric space $X$.
A \emph{pre-orientation of $\G$} is a function that to any $x \in X$ assigns an orientation $\o(x)$ of $\G(x)$.
A pre-orientation (of $\G$) is an \emph{orientation} if it is \emph{continuous}, in the sense that, given any $x_* \in X$, there exist $K \in \o(x_*)$ and a neighborhood $W$ of $x_*$ such that $K \in \o(x)$ for all $x \in W$.
The map $\G$ is said to be \emph{orientable} if it admits an orientation, and \emph{oriented} if an orientation has been chosen.
In particular, a subset $Y$ of $\Phi_0(E,F)$ is orientable or oriented if so is the inclusion map $Y \hookrightarrow \Phi_0(E,F)$.
\end{definition}

Observe that the set $\hat\Phi_0(E,F)$ of the oriented operators of $\Phi_0(E,F)$ has a natural topology, and the natural projection $\pi\colon\hat\Phi_0(E,F) \to \Phi_0(E,F)$ is a $2$-fold covering space (see \cite{BeFu2} for details).
Therefore, an orientation of a map $\G$ as in Definition \ref{orientation of Gamma} could be regarded as a lifting $\hat \G$ of $\G$.
This implies that, if the domain $X$ of $\G$ is simply connected and locally path connected, then $\G$ is orientable.

\medskip
Let $f\colon U \to F$ be a $C^1$-map defined on an open subset of $E$, and
denote by $df_x \in \L(E,F)$ the Fr\'echet differential of $f$ at a point $x \in U$.

We recall that $f$ is said to be \emph{Fredholm of index $n$}, from now on written $f \in \Phi_n$, if $df_x \in \Phi_n(E,F)$ for all $x \in U$.
Therefore, if $f \in \Phi_0$, Definition \ref{orientation of Gamma} and the continuity of the differential map $df\colon U \to \Phi_0(E,F)$ suggest the following

\begin{definition}[Orientation of a $\Phi_0$-map in Banach spaces]
\label{Orientation of a map in the flat case}
Let $U$ be an open subset of $E$ and $f\colon U \to F$ a Fredholm map of index zero.
A \emph{pre-orientation} or an \emph{orientation} of $f$ are, respectively, a pre-orientation or an orientation of $df$, according to Definition \ref{orientation of Gamma}.
The map $f$ is said to be \emph{orientable} if it admits an orientation, and \emph{oriented} if an orientation has been chosen.
\end{definition}

\begin{remark}
\label{double nature of a Phi-zero operator}
A very special $\Phi_0$-map is given by an operator $T \in \Phi_0(E,F)$.
Thus, for $T$ there are two different notions of orientations: the algebraic one, according to Definition \ref{orientation of T}; and the one regarding $T$ as a $C^1$-map (according to Definition \ref{Orientation of a map in the flat case}).
In each case $T$ admits exactly two orientations (in the second one this is due to the connectedness of the domain $E$).
Hereafter, we shall tacitly assume that the two notions agree. Namely, $T$ has an algebraic orientation $\o$ if and only if its differential $dT_x\colon \dot x \mapsto T\dot x$ has the $\o$ orientation for all $x \in E$.
\end{remark}

We will show how the notion of orientation in Definition \ref{Orientation of a map in the flat case} can be extended to the case of maps acting between real Banach manifolds.
To this purpose, we need some further notation.

For short, by a \emph{manifold} we shall mean a smooth Banach manifold embedded in a real Banach space.

Given a manifold $\M$ and a point $x \in \M$, the tangent space of $\M$ at $x$ will be denoted by $T_x\M$.
If $\M$ is embedded in a Banach space $\widetilde E$, $T_x\M$ will be identified with a closed subspace of $\widetilde E$, for example by regarding any tangent vector of $T_x\M$ as the derivative $\g'(0)$ of a smooth curve $\g\colon (-1,1) \to \M$ such that $\g(0)=x$.

\medskip
Assume that $f \colon \M \to \N$ is a $C^1$-map between two manifolds, respectively embedded in $\widetilde E$ and $\widetilde F$ and modelled on $E$ and $F$.
As in the flat case, $f$ is said to be \emph{Fredholm of index $n$} (written $f \in \Phi_n$) if so is the differential $df_x \colon T_x\M \to T_{f(x)}\N$, for any $x \in \M$.

Given $f \in \Phi_0$, suppose that to any $x \in \M$ it is assigned an orientation $\o(x)$ of $df_x$ (also called \emph{orientation of $f$ at $x$}).
As above, the function $\o$ is called a \emph{pre-orientation} of $f$, and an \emph{orientation} if it is continuous, in a sense to be specified (see Definition \ref{Orientation of a map in the non-flat case}).

\begin{definition}
\label{pre-oriented composition}
The pre-oriented composition of two (or more) pre-oriented maps between manifolds is given by assigning, at any point $x$ of the domain of the composite map, the composition of the orientations (according to Definition \ref{oriented composition}) of the differentials in the chain representing the differential at $x$ of the composite map.
\end{definition}

Assume that $f \colon \M \to \N$ is a $C^1$-diffeomorphism.
Then, in particular, given any $x \in \M$, the differential $df_x$ is an isomorphism.
Thus, for any $x \in \M$, we may take as $\o(x)$ the natural orientation of $df_x$ (recall Definition \ref{natural orientation}).
This pre-orientation of $f$ turns out to be continuous according to Definition \ref{Orientation of a map in the non-flat case} below (it is, in some sense, constant).
From now on, unless otherwise stated, \textbf{any diffeomorphism will be considered oriented with the natural orientation}.
In particular, in a composition of pre-oriented maps, all charts and parametrizations of a manifold will be tacitly assumed to be naturally oriented.

\begin{definition}[Orientation of a $\Phi_0$-map between manifolds]
\label{Orientation of a map in the non-flat case}
Let $f \colon \M \to \N$ be a $\Phi_0$-map between two manifolds modelled on $E$ and $F$, respectively.
A pre-orientation of $f$ is an \emph{orientation} if it is \emph{continuous} in the sense that, given any two charts, $\varphi\colon U \to E$ of $\M$ and $\psi\colon V \to F$ of $\N$, such that $f(U) \subseteq V$, the pre-oriented composition
\[
\psi \circ f \circ \varphi^{-1} \colon U \to V
\]
is an oriented map according to Definition \ref{Orientation of a map in the flat case}.

The map $f$ is said to be \emph{orientable} if it admits an orientation, and \emph{oriented} if an orientation has been chosen.
\end{definition}

\medskip
Perhaps, the simplest example of non-orientable $\Phi_0$-map is given by a constant map from the $2$-dimensional projective space into $\R^2$ (see \cite{BeFu2}).

\begin{remark}
\label{composition of orientations}
One can check that the pre-oriented composition of orientations is an orientation.
\end{remark}

\begin{remark}
\label{adapted attributes}
Regarding the attribution that we will assign to some particular orientations of $\Phi_0$-maps between manifolds, whenever it makes sense, we will adapt the terminology for $\Phi_0$-operators, such as ``natural orientation'', ``associated orientation'', ``canonical orientation''.
\end{remark}
For example any local diffeomorphism $f\colon \M \to \N$ admits the \emph{natural orientation}, given by assigning the natural orientation to the operator $df_x$, for any $x \in \M$ (see Definition \ref{natural orientation}).
As another example, assume the manifolds $\M$ and $\N$ have the same finite dimension and are oriented, then any $C^1$-map between them admits the \emph{associated orientation} (see Definition~\ref{associated orientation}).
A third example is given by a $C^1$-map $f\colon \R^k \to \R^k$: it can be given the \emph{canonical orientation} (see Definition \ref{canonical orientation}).

\medskip
The concept of canonical orientation of a $C^1$-map $f\colon \R^k \to \R^k$ can be extended to a more general situation that we shall need in the next section.
In fact, if $E$ is a real Banach space, in spite of the fact that the function $\det\colon \A(E) \to \R$ can be discontinuous (see e.g.~\cite{BeFuPeSp07}), one has the following

\begin{remark}
\label{canonical orientation of a map}
Let $X$ be a metric space and $E = E_1 \per E_2$ a real Banach space,
with $\dim E_2 < +\infty$.
Assume that $\G\colon X \to \A(E)$ is a continuous map that can be represented in a block matrix form as follows:
\[
\G =
\left(
\begin{array}{cc}
I_{11} & \G_{12}\\[2ex]
0 & \G_{22}
\end{array}
\right),
\]
where $I_{11}$ is the identity of $E_1$, $\G_{12}\colon X \to \L(E_2,E_1)$, and $\G_{22}\colon X \to \L(E_2)$.
Then, according to Remark \ref{determinant}, one has $\det \G(x) = \det \G_{22}(x)$, for all $x \in X$.
Moreover, the pre-orientation of $\G$ given by assigning, to any $x \in X$, the canonical orientation of the operator $\G(x)$ is actually an orientation, and has the property that $\sign \G(x) = \sign \det \G(x)$ for all $x \in X$.
\end{remark}

\medskip
Similarly to the case of a single map, one can define a notion of orientation of a continuous family of $\Phi_0$-maps depending on a parameter $s \in [0,1]$.
To be precise, one has the following

\begin{definition}[Oriented $\Phi_0$-homotopy]
\label{Phi-zero-homotopy}
A \emph{$\Phi_0$-homotopy} between two Banach manifolds $\M$ and $\N$ is a $C^1$-map $h \colon [0,1] \per\M \to \N$ such that, for any $s \in [0,1]$, the partial map $h_s= h(s,\cdot)$ is Fredholm of index zero.
An orientation of $h$ is a \emph{continuous function} $\o$ that to any $(s,x) \in [0,1]\per\M$ assigns an orientation $\o(s,x)$ to the differential
$d(h_s)_x \in \Phi_0(T_x\M, T_{h(s,x)}\N)$.
Where ``continuous'' means that, given any chart $\varphi\colon U \to E$ of $\M$, a subinterval $J$ of $[0,1]$, and a chart $\psi\colon V \to F$ of $\N$ such that $h(J\per U) \subseteq V$, the pre-orientation of the map $\G\colon J\per U \to \Phi_0(E,F)$ that to any $(s,x) \in J\per U$ assigns the pre-oriented composition
\[
d(\psi \circ h_s \circ \varphi^{-1})_x
= d\psi_{h(s,x)}d(h_s)_x (d\varphi_x)^{-1}
\]
is an orientation, according to Definition \ref{orientation of Gamma}.

The homotopy $h$ is said to be \emph{orientable} if it admits an orientation, and \emph{oriented} if an orientation has been chosen.
\end{definition}

If a $\Phi_0$-homotopy $h$ has an orientation $\o$, then any partial map $h_s = h(s,\cdot)$ has a \emph{compatible} orientation $\o(s,\cdot)$.
Conversely, on has the following

\begin{proposition}[\!\cite{BeFu1,BeFu2}]
\label{orientation transport}
Let $h\colon [0,1]\per\M \to \N$ be a $\Phi_0$-homotopy, and assume that one of its partial maps, say $h_s$, has an orientation.
Then, there exists and is unique an orientation of $h$ which is compatible with that of $h_s$.
In particular, if two maps from $\M$ to $\N$ are $\Phi_0$-homotopic, then they are both orientable or both non-orientable.
\end{proposition}

As a consequence of Proposition \ref{orientation transport}, one gets that any $C^1$-map $f\colon \M \to \M$ which is $\Phi_0$-homotopic to the identity is orientable, since so is the identity (even when $\M$ is finite dimensional and not orientable).

\medskip
The degree for oriented $\Phi_0$-maps defined in \cite{BeFu1} satisfies the three fundamental properties stated below and called \emph{Normalization, Additivity and Homotopy Invariance}.
By an axiomatic approach similar to the one due to Amann-Weiss in \cite{Amann-Weiss} for the Leray--Schauder degree, in \cite{BeFu5} it is shown that the degree constructed in \cite{BeFu1} is the only possible integer valued function that satisfies these three properties.

To be more explicit, let us define, first, the domain of this degree function.
Given an oriented $\Phi_0$-map $f\colon \M \to \N$, an open (possibly empty) subset $U$ of $\M$, and a \emph{target value} $y \in \N$, the triple $(f,U,y)$ is said to be \emph{admissible} for the degree provided that $U \cap f^{-1}(y)$ is compact.
From the axiomatic point of view, the degree is an integer valued function, $\deg$, defined on the class of all the admissible triples, that satisfies the following three \emph{fundamental properties}:

\medskip
\begin{itemize}
\item
(Normalization) \emph{If $f\colon \M \to \N$ is a naturally oriented diffeomorphism onto an open subset of $\N$, then
\[
\deg(f,\M,y) = 1,\quad \forall y \in f(\M).
\]}
\item
(Additivity) \emph{Let $(f,U,y)$ be an admissible triple.
If $U_1$ and $U_2$ are two disjoint open subsets of $U$ such that $U \cap f^{-1}(y) \subseteq U_1 \cup U_2$, then}
\[
\deg(f,U,y) = \deg(f|_{U_1},U_1,y) + \deg(f|_{U_2},U_2,y).
\]
\item
(Homotopy Invariance) \emph{Let $h\colon [0,1]\per\M \to \N$ be an oriented $\Phi_0$-homotopy, and $\g\colon [0,1] \to \N$ a continuous path.
If the set
\[
\big\{(s,x) \in [0,1]\per\M: h(s,x) = \g(s)\big\}
\]
is compact, then
$
\deg(h(s,\cdot),\M,\g(s))
$
does not depend on $s \in [0,1]$.}
\end{itemize}

\medskip
Other properties can be deduced from the fundamental ones (see \cite{BeFu5} for details). We mention only some of them. One of these is the
\medskip
\begin{itemize}
\item
(Localization) \emph{If $(f,U,y)$ is an admissible triple, then
\[
\deg(f,U,y) = \deg(f|_U,U,y).
\]}
\end{itemize}

Another one is the
\medskip
\begin{itemize}
\item
(Excision) \emph{If $(f,U,y)$ is admissible and $V$ is an open subset of $U$ such that $f^{-1}(y)\cap U \subseteq V$, then
\[
\deg(f,U,y) = \deg(f,V,y).
\]}
\end{itemize}

A significant one is the
\medskip
\begin{itemize}
\item
(Existence) \emph{If $(f,U,y)$ is admissible and $\deg(f,U,y) \not= 0$, then the equation $f(x) = y$ admits at least one solution in $U$.}
\end{itemize}

\medskip
Roughly speaking, given an admissible triple $(f,U,y)$, the integer $\deg(f,U,y)$ is an algebraic count of the solutions in $U$ of the equation $f(x) = y$.
More precisely, as a consequence of the fundamental properties, one gets the following

\medskip
\begin{itemize}
\item
(Computation Formula)
\emph{If $(f,U,y)$ is admissible and $y$ is a regular value for $f$ in $U$, then the set $U \cap f^{-1}(y)$ is finite and
\[
\deg(f,U,y) = \sum_{x \in U \cap f^{-1}(y)} \sign(df_x).
\]}
\end{itemize}

Another property that can be deduced from the fundamental ones is a generalization of the Homotopy Invariance Property, that we will need in Section \ref{Results 2}.
This requires the following extension of the concept of $\Phi_0$-homotopy:

\begin{definition}[Extended $\Phi_0$-homotopy]
\label{extended Phi-zero-homotopy}
An \emph{extended $\Phi_0$-homotopy} from $\M$ to $\N$ is a \hbox{$C^1$-}map $h\colon \I\per\M \to \N$, where $\I$ is an arbitrary (nontrivial) real interval, such that any partial map $h_s = h(s,\cdot)$ of $h$ is a $\Phi_0$-map.
\end{definition}

The notion of orientation for an extended $\Phi_0$-homotopy is practically identical to the one in Definition \ref{Phi-zero-homotopy} and its formulation is left to the reader.

\medskip
As a consequence of the Excision and the Homotopy Invariance properties of the degree we get the following

\medskip
\begin{itemize}
\item
(Generalized Homotopy Invariance)
\emph{Let $h\colon \I\per\M \to \N$ be an oriented extended $\Phi_0$-homotopy, $\g\colon \I \to N$ a continuous path, and $W$ an open subset of $\I\per\M$.
Given any $s \in \I$, denote by $W_s = \{x \in \M: (s,x) \in W\}$ the $s$-slice of $W$.
If the set
$
\big\{(s,x) \in W\colon h(s,x) = \g(s)\big\}
$
is compact, then
$
\deg(h_s,W_s,\g(s))
$
does not depend on $s \in \I$.}
\end{itemize}

\medskip
The easy proof of this property can be performed by showing that the integer valued function $s \in \mathcal I \mapsto \nu(s) := \deg(h_s,W_s,\g(s))$ is locally constant.
In fact, given any $s_* \in \I$, because of the compactness of the set
\[
\big\{(s,x) \in W\colon h(s,x) = \g(s)\big\},
\]
one can find a box $J \per V \subseteq W$, with V open in $\M$ and $J$ an open interval containing $s_*$, such that $W_s \cap h_s^{-1}(\g(s)) \subseteq V$ for all $s \in J\cap\I$.
Thus, from the Excision Property, one gets $\nu(s) = \deg(h_s,V,\g(s))$ for all $s \in J\cap\I$. Moreover, because of the Homotopy Invariance Property, $\nu(s)$ does not depend on $s \in J\cap\I$.
Hence, since $\I$ is connected and $s_* \in \I$ is arbitrary, one gets the assertion.

\section{The eigenvalue problem and the associated topological degree}
\label{Results 1}

Hereafter, $G$ and $H$ will be two real Hilbert spaces, with inner product and norm denoted by $\langle \cdot,\cdot \rangle$ and $\|\cdot\|$, respectively.

Consider the eigenvalue problem
\begin{equation}
\label{eigenvalue problem}
\left\{
\begin{aligned}
&Lx = \l Cx\\
&x \in \S,
\end{aligned}\right.
\end{equation}
where $\l \in \R$, $L, C \in \L(G,H)$, $C$ is compact, and $\S$ is the unit sphere of $G$.
We assume that the operator $L - \l C \in \L(G,H)$ is invertible for some $\l \in \R$.
Therefore, because of the compactness of $C$, $L - \l C$ is Fredholm of index zero for any $\l \in \R$ (recall property (6) of Fredholm operators in Subsection \ref{Topological preliminaries}).
In particular, $\Ker(L - \l C)$ is always finite dimensional, and nontrivial if and only if $\l \in \R$ is an eigenvalue of \eqref{eigenvalue problem}.
Moreover, the set of all the real eigenvalues of \eqref{eigenvalue problem} is discrete.

\medskip
A pair $(\l,x)$ belonging to the cylinder $\R\per\S$ will be called an \emph{eigenpoint of \eqref{eigenvalue problem}} if it satisfies the equation $Lx = \l Cx$.
In this case, $x$ is a \emph{unit eigenvector} of \eqref{eigenvalue problem} corresponding to the eigenvalue $\l$.

The set of the eigenpoints of \eqref{eigenvalue problem} will be denoted by $\mathcal S$.
Hence, given any $\l \in \R$, the \emph{$\l$-slice} $\mathcal S_\l = \{x \in \S: (\l,x) \in \mathcal S\}$ of $\mathcal S$ coincides with $\S \cap \Ker(L-\l C)$.

 Observe that $\mathcal S_\l$ is nonempty if and only if $\l$ is an eigenvalue of \eqref{eigenvalue problem}.
 In this case $\mathcal S_\l$ will be called the \emph{eigensphere of \eqref{eigenvalue problem} corresponding to $\l$}.
 In fact, it is a sphere whose (finite) dimension equals that of $\Ker(L-\l C)$ minus one.

If $\l \in \R$ is an eigenvalue of \eqref{eigenvalue problem}, the nonempty subset $\{\l\} \per \mathcal S_\l$ of $\mathcal S$ will be called the \emph{$\l$-eigenset of \eqref{eigenvalue problem}}.

\medskip
It is convenient to regard $\R\per\S$ as the subset of the Hilbert space $\R\per G$ satisfying the equation $g(\l,x) = 1$, where $g \colon \R\per G \to \R$ is defined by $g(\l, x) = \langle x, x \rangle$.
The differential $dg_p \in \L(\R\per G,\R)$ of $g$ at a point $p = (\l, x)$ is given by $(\dot\l,\dot x) \mapsto 2\langle x,\dot x \rangle$. Therefore, the set of the critical points of $g$ is the $\l$-axis $ x=0$ and, consequently, the number $1$ is a regular value for $g$. This shows that $\R\per\S$ is a smooth manifold of codimension one in $\R\per G$ and, given any $p = (\l, x) \in g^{-1}(1)$, the tangent space of $\R\per\S$ at $p$ is the kernel of $dg_p$, namely
\[
T_{(\l, x)}(\R\per\S) = \big\{(\dot\l,\dot x) \in \R\per G: \langle x,\dot x\rangle = 0 \big\} = \R\per x^\perp.
\]

Observe that, if $\dim G = 1$ and $(\l, x) \in \R\per\S$, then $x^\perp = \{0\}$ and the tangent space $T_{(\l, x)}(\R\per\S)$ is the subspace $\R\per\{0\}$ of $\R\per\R$.
Moreover, the cylinder $\R\per\S$ is disconnected: it is the union of two horizontal lines, $\R\per\{-1\}$ and $\R\per\{1\}$.
Due to this fact, in order to write some statements in a simpler form, \textbf{hereafter, unless otherwise stated, we will assume that the dimensions of the Hilbert spaces $G$ and $H$ are bigger than $1$}, so that the cylinder $\R\per\S$ is connected.

\medskip
Define the smooth map
\[
\overline \Psi \colon \R\per G \to H
\quad
\text{by}
\quad
(\l, x) \mapsto L x-\l Cx
\]
and observe that it is Fredholm of index one.
Therefore, its restriction
\[
\Psi \colon \R \per \S \to H
\]
to the $1$-codimensional submanifold $\R \per \S$ of $\R \per G$ is Fredholm of index zero.
To see this, recall that the differential of $\Psi$ at $p \in \R\per\S$ is the restriction of $d \overline \Psi_p$ to the $1$-codimensional subspace $T_p(\R\per\S)$ of $\R\per G$.

The map $\Psi$ will play a fundamental role in this paper.
Notice that its zeros are the eigenpoints of \eqref{eigenvalue problem}.
That is, $\mathcal S = \Psi^{-1}(0)$.

We point out that $\Psi$ is orientable and, because of the connectedness of the manifold $\R\per\S$, admits exactly two orientations.
In fact, in the finite dimensional case, an orientation of $\Psi$ is equivalent to a pair of orientations, one of the domain and one of the codomain, up to an inversion of both of them (see Definition \ref{associated orientation}); while, if $\dim G= +\infty$, the orientability of $\Psi$ is a consequence of the fact that the cylinder $\R \per \S$ is simply connected (it is actually contractible).
Therefore, \textbf{from now on, we shall assume that $\Psi$ is oriented}.
No matter which one of the two orientations one selects, all the statements in this paper hold true.

\begin{definition}
\label{isolated set}
Let $X$ be a metric space and $\K \subseteq \A \subseteq X$. We shall say that \emph{$\K$ is an isolated subset of $\A$} if it is compact and relatively open in $\A$.
Thus, there exists an open subset $U$ of $X$ such that $U \cap \A = \K$, which will be called an \emph{isolating neighborhood of $\K$ among (the elements of) $\A$}.
\end{definition}

\begin{definition}
\label{definition of Psi-degree}
Let $\K \subset \R\per\S$ be an isolated subset of $\Psi^{-1}(0)$.
By the \emph{\hbox{$\Psi$\!-degree} of $\K$} we mean the integer $\psideg(\K) := \deg(\Psi,U,0)$, where $U \subseteq \R\per\S$ is any isolating neighborhood of $\K$ among $\Psi^{-1}(0)$.
If $p$ is an isolated zero of $\Psi$, we shall simply write $\psideg(p)$ instead of $\psideg(\{p\})$.
\end{definition}

Notice that this definition is well-posed, thanks to the Excision Property of the degree.

\begin{remark}
\label{regular zero of Psi}
If $p \in \Psi^{-1}(0)$ is such that the differential $d\Psi_p$ is an isomorphism, then, as a consequence of the Local Inverse Function Theorem, $\Psi$ maps diffeomorphically a neighborhood of $p$ in $\R\per\S$ onto a neighborhood of $0$ in $H$. Thus, $\{p\}$ is isolated among $\Psi^{-1}(0)$ and, because of the Computation Formula of the degree (see Section \ref{Preliminaries}), $\psideg(p) = \sign(d\Psi_p)$, which is either $1$ or $-1$, depending on whether or not the orientation of $d\Psi_p$ is the natural one.
\end{remark}

\begin{definition}
\label{simple}
An eigenpoint $(\l_*,x_*)$ of \eqref{eigenvalue problem} will be called \emph{simple} if the associated $\Phi_0$-operator $T = L-\l_* C$ satisfies the following conditions:
\begin{itemize}
\item[(1)]
$\Ker T = \R x_*$,
\item[(2)]
$Cx_*\notin \Im T$.
\end{itemize}
\end{definition}

Notice that, if $p_*=(\l_*,x_*)$ is a simple eigenpoint, then the eigenset $\{\l_*\}\per\mathcal S_{\l_*}$ has only two elements: $p_*$ and its \emph{twin eigenpoint} $\bar p_*=(\l_*,-x_*)$, which is as well simple.
Moreover, since $T = L-\l_* C$ is Fredholm of index zero, its image has codimension one in $H$. Therefore one has the following

\begin{remark}
\label{splitting of H}
If $(\l_*,x_*)$ is a simple eigenpoint of \eqref{eigenvalue problem}, then $H = \Im T \oplus \R Cx_*$.
Thus, $\l_*$ is a simple eigenvalue of the equation $Lx = \l Cx$.
\end{remark}

\medskip
The following result is the key to the proof of Theorem \ref{ex conjecture}.
Despite the fact that we have assumed $\dim G > 1$, Theorem \ref{ex conjecture} holds true in any dimension: its assertion in the $1$-dimensional case will be verified in Example \ref{example1}.

\begin{theorem}
\label{psi-degree simple eigenpoint}
Let $p_* = (\l_*,x_*)$ and $\bar p_*=(\l_*,-x_*)$ be two simple twin eigenpoints of \eqref{eigenvalue problem}.
Then, $\psideg(p) = \psideg(\bar p) = \pm 1$.
Consequently, the $\psidegree$ of the $\l_*$-eigenset $\{\l_*\} \per \mathcal S_{\l_*}$ is non-zero.
\end{theorem}

\begin{proof}
It is enough to prove that $\psideg(p) = \psideg(\bar p) = \pm 1$:
the last assertion follows from the Additivity Property of the degree.

Since $p_*$ is simple, the $\l_*$-eigensphere $\mathcal S_{\l_*}$ of \eqref{eigenvalue problem} consists of two antipodal points: $x_*$ and $-x_*$.
Both the tangent spaces of $\S$ at these points coincide with the $1$-codimensional subspace $x_*^\perp$ of $G$.
Thus, the tangent spaces of the cylinder $\R\per\S$ at the twin eigenpoints $p_*$ and $\bar p_*$ are equal to the $1$-codimensional subspace $\R\per x_*^\perp$ of the Hilbert space $\R\per G$.
The differentials $d\Psi_{p_*}$ and $d\Psi_{\bar p_*}$ (acting from $\R\per x_*^\perp$ to $H$) are given, respectively, by
\[
(\dot \l, \dot x) \mapsto T \dot x - \dot\l Cx_*
\quad
\text{and}
\quad
(\dot \l, \dot x) \mapsto T \dot x + \dot\l Cx_*,
\]
where $T$, as in Definition \ref{simple}, denotes the operator $L-\l_*C$.

As one can check (see, for example, \cite[Lemma 3.2]{BeCaFuPe-s4})), the fact that the eigenpoints $p_*$ and $\bar p_*$ are simple implies that the differentials $d\Psi_{p_*}$ and $d\Psi_{\bar p_*}$ are invertible.
Consequently, according to Remark \ref{regular zero of Psi}, the $\Psi$\hbox{-}degrees of $p_*$ and $\bar p_*$ coincide, respectively, with $\sign(d\Psi_{p_*})$ and $\sign(d\Psi_{\bar p_*})$, which are both non-zero.

Thus, it remains to prove that these two signs are equal, which means that the orientations of $\Psi$ at these points are both natural or both not natural.

For this purpose, it is convenient to fix an orientation of $\Psi$ at one of the two eigenpoints $p_*$ and $\bar p_*$ (for example by choosing the natural orientation of $d\Psi_{\bar p_*}$) and \emph{to transport it, continuously, along a curve, up to the other eigenpoint}.

A suitable curve is a \emph{$\l_*$-meridian}.
That is, a geodesic in $\R\per\S$ of the type
\[
\mathcal G = \big\{(\l_*,x) \in \R\per\S: x = \sin\t\, x_* + \cos\t\, x_e, \; \t \in [-\pi/2,\pi/2]\big\},
\]
where $x_e$ is an element of the \emph{equator} $\S\cap x_*^\perp$ of $\S$ (recall that $\dim G > 1$) and $\t$ may be regarded as a \emph{latitude}.

Having chosen $x_e$, and the consequent $\l_*$-meridian, we will ``observe'' the differential of $\Psi$, along $\mathcal G$, from the point of view of a self-map defined on a convenient ``flat space'';
namely, the Hilbert space $G_*\per\R\per\R$, where $G_*$ is the $2$-codimensional subspace $x_*^\perp \cap x_e^\perp$ of $G$, which is nonempty because of the assumption $\dim G > 1$, even if trivial when $\dim G = 2$.

We will ``observe'' the map $\Psi$ by means of a convenient composition $\widetilde \Psi$ = \hbox{$\s^{-1}\!\circ\Psi\circ\eta$}, where $\eta\colon W \to \R\per\S$ is a parametrization (i.e.\ the inverse of a chart) defined on an open subset $W$ of $G_*\per\R\per\R$ and $\s\colon G_*\per\R\per\R \to H$ is an invertible bounded linear operator (a global parametrization of $H$).

Let us define $\eta$, first.
Consider the open subset
\[
W = \big\{(y,\t,\l) \in B\per(-\pi,\pi)\per\R\big\}
\]
of $G_*\per\R\per\R$, where $B$ stands for the open unit ball of $G_*$, and let
$\eta \colon W \to \R\per\S$ be the map given by
\[
\eta(y,\t,\l) = \big(\l, y + \sqrt{1-\|y\|^2}(\sin\t\,x_* + \cos\t\,x_e)\big).
\]
Notice that, under $\eta$, the eigenpoints $\bar p_*$ and $p_*$ correspond to $\bar u_*=(0,-\pi/2, \l_*)$ and $u_*=(0,\pi/2,\l_*)$, respectively.
One can check that $\eta$ is a diffeomorphism onto an open subset of $\R\per\S$ containing the meridian $\mathcal G$.

We now define $\s$.
From Remark \ref{splitting of H} we get the splitting
\begin{equation}
\label{splitting}
H = T(x_*^\perp)\oplus\R Cx_* = (T(G_*)\oplus\R Tx_e)\oplus\R Cx_*.
\end{equation}
Thus, $H$ can be identified with $G_*\per\R\per\R$ by means of the isomorphism
\[
\s\colon G_*\per\R\per\R \to H, \quad (y,a,b) \mapsto Ty+aTx_e+bCx_*.
\]

We assume that $\eta$ and $\s^{-1}$ are naturally oriented and that $\widetilde \Psi$ has the composite orientation.
Therefore, recalling Corollary \ref{sign of composition},
\begin{equation}
\label{sign equality}
\sign(d\widetilde\Psi_{\bar u_*}) = \sign(d\Psi_{\bar p_*}) \quad \text{and} \quad \sign(d\widetilde\Psi_{u_*}) = \sign(d\Psi_{p_*}),
\end{equation}
whatever the orientation of $\Psi$.

Hence, it remains to prove that $\sign(d\widetilde\Psi_{\bar u_*}) = \sign(d\widetilde\Psi_{u_*})$, no matter what is the orientation of $\widetilde \Psi$.

To this purpose, consider the straight path $\g\colon [-\pi/2,\pi/2] \to W$ defined by $\g(\t) = (0,\t,\l_*)$.
This path joins $\g(-\pi/2) = \bar u_*$ with $\g(\pi/2) = u_*$,
therefore it is suitable for the continuous transport of the orientation of $\widetilde\Psi$ from $\bar u_*$ to $u_*$.
Notice that the image of the simple arc $\t \mapsto \eta(\g(\t))$ is the $\l_*$-meridian $\mathcal G$.

Taking into account that $\Psi(\l,x) = Tx - (\l-\l_*)Cx$ and that $Tx_* = 0$, given any $\t \in [-\pi/2,\pi/2]$, we get
\[
d(\Psi\circ\eta)_{\g(\t)}(\dot y,\dot\t,\dot\l) =
T\dot y - \dot\t\sin\t\, Tx_e - \dot\l\sin\t\,Cx_* - \dot\l\cos\t\,Cx_e.
\]
Therefore, recalling that $\s^{-1}$, being linear, coincides with its differential, we obtain
\[
d(\widetilde\Psi)_{\g(\t)}(\dot y,\dot\t,\dot\l) =
\s^{-1}\big(T\dot y - \dot\t\sin\t\, Tx_e - \dot\l\sin\t\,Cx_* - \dot\l\cos\t\,Cx_e\big).
\]
Since, according to the splitting \eqref{splitting}, $Cx_e$ can be written as $Ty_* + \a Tx_e + \b Cx_*$ for some $y_* \in G_*$ and $\a,\b \in \R$, we have
\[
d(\widetilde\Psi)_{\g(\t)}(\dot y,\dot\t,\dot\l) =
\big(\dot y, - \dot\t\sin\t, - \dot\l\sin\t\big) - \dot\l\cos\t\big(y_*,\a,\b\big),
\]
which can be represented as
\[
d\widetilde\Psi_{\g(\t)}(\dot y, \dot \t, \dot \l) =
\left(
\begin{array}{ccc}
I_{11} & 0 & -\cos\t\,y_*\\[2ex]
0 & -\sin\t & -\a\cos\t\\[2ex]
0 & 0 & -(\sin\t+\b\cos\t)
\end{array}
\right)
\left(
\begin{array}{c}
\dot y\\[2ex]
\dot\t \\[2ex]
\dot\l
\end{array}
\right),
\]
where $I_{11}$ is the identity on $G_*$.

Thus, the continuous map $\G\colon [-\pi/2,\pi/2] \to \A(G_*\per\R\per\R)$, given by $\t \mapsto d\widetilde\Psi_{\g(\t)}$, is in block matrix form as in Remark \ref{canonical orientation of a map}.
Consequently, up to an inversion of the orientation of $\Psi$, we may assume that $\G$ has the canonical orientation, which is such that
\[
\sign(d\widetilde\Psi_{\g(\t)}) = \sign\det(d\widetilde\Psi_{\g(\t)})
= \sign\big(\sin\t(\sin\t+\b\cos\t)\big).
\]
Recalling that $\bar u_* = \g(-\pi/2)$ and $u_* = \g(\pi/2)$, we finally obtain $\sign(d\widetilde\Psi_{\bar u_*}) = \sign(d\widetilde\Psi_{u_*})$, and the assertion ``$\psideg(p) = \psideg(\bar p) = \pm 1$'' follows from \eqref{sign equality}.
\end{proof}

\section{The perturbed eigenvalue problem and global continuation}
\label{Results 2}

Given, as before, two real Hilbert spaces $G$ and $H$, consider the problem
\begin{equation}
\label{perturbed eigenvalue problem}
\left\{
\begin{aligned}
&Lx + s N(x) = \l Cx\\
&x \in \S,
\end{aligned}\right.
\end{equation}
where $s, \l \in \R$, $L, C \in \L(G,H)$, $C$ is compact, and $N\colon \S \to H$ is a $C^1$ compact map defined on the unit sphere of $G$.
As in the unperturbed problem \eqref{eigenvalue problem}, we assume that $L - \l C$ is invertible for some $\l \in \R$.

\medskip
A triple $(s,\l,x) \in \R\per\R\per\S$ is a \emph{solution}
of \eqref{perturbed eigenvalue problem} if it satisfies the equation $Lx + s N(x) = \l Cx$.
The third element $x \in \S$ is said to be a \emph{unit eigenvector} corresponding to the \emph{eigenpair} $(s,\l)$.
The set of all the solutions of \eqref{perturbed eigenvalue problem} is denoted by $\Sigma$, while $\mathcal E$ stands for the subset of $\R^2$ of the eigenpairs of \eqref{perturbed eigenvalue problem}.

Observe that $\mathcal E$ is the projection of $\Sigma$ into the $s\l$-plane, and that the $s=0$ slice $\Sigma_0$ of $\Sigma$ coincides with the set $\mathcal S = \Psi^{-1}(0)$ of the eigenpoints of \eqref{eigenvalue problem}.

A solution of \eqref{perturbed eigenvalue problem} is said to be \emph{trivial} if it is of the type $(0,\l,x)$.
In this case $p=(\l,x)$ is the \emph{corresponding eigenpoint} (of the unperturbed problem \eqref{eigenvalue problem}).
When $p$ is simple, the solution $(0,\l,x)$ will be as well said to be \emph{simple}.
Therefore, any simple solution is trivial, but not viceversa.

\smallskip
Consider the $C^1$-map
\[
\Psi^+\colon \R\per\R\per\S \to H, \quad (s,\l,x) \mapsto \Psi(\l,x) + sN(x),
\]
where, we recall, $\Psi\colon \R\per\S \to H$ is defined by $\Psi(\l,x) = Lx - \l Cx$.
Notice that the zeros of $\Psi^+$ are the solutions of \eqref{perturbed eigenvalue problem}; that is, $\Sigma = (\Psi^+)^{-1}(0)$.

Since $\Psi$ is Fredholm of index zero, because of the compactness of $N$, any partial map $\Psi^+_s = \Psi^+(s, \cdot,\cdot)$ of $\Psi^+$ is a $\Phi_0$-map from $\R\per\S$ to $H$.
In fact, in Banach spaces, the differential at any point of the domain of a compact $C^1$-map is a compact linear operator, and this holds true also in Banach manifolds, due to the fact that they are locally diffeomorphic to open sets of Banach spaces.

Therefore, according to Definition \ref{extended Phi-zero-homotopy}, $\Psi^+$ is an extended $\Phi_0$-homotopy from $\R\per\S$ into $H$.

Due to the fact that the partial map $\Psi^+_0$ of $\Psi^+$ coincides with the oriented map $\Psi$, thanks to Proposition \ref{orientation transport}, the extended $\Phi_0$-homotopy $\Psi^+$ admits an orientation (a unique one) which is compatible with that of $\Psi$.
Therefore, from now on, $\Psi^+$ will be considered an oriented extended $\Phi_0$-homotopy.

\medskip
Since the set $\Sigma = (\Psi^+)^{-1}(0)$ has a distinguished (trivial) subset, namely $\{0\}\per\Sigma_0$, it makes sense to consider the notion of bifurcation point.
A trivial solution $q_* =(0,\l_*,x_*)$ of \eqref{perturbed eigenvalue problem} is a \emph{bifurcation point} provided that any neighborhood of $q_*$ contains nontrivial solutions.

A bifurcation point $q_*$ is said to be \emph{global} (in the sense of Rabinowitz \cite{Ra}) if there exists a connected set of nontrivial solutions whose closure contains $q_*$ and it is either unbounded or meets a bifurcation point $q^*$ different from $q_*$.

Particularly meaningful is the study of bifurcation points belonging to a set of trivial solutions of the type $\{0\}\per\{\l_*\}\per\mathcal S_{\l_*}$, whose eigensphere $\mathcal S_{\l_*}$ is \emph{nontrivial} (that is, with positive dimension).
Since, in this case, $0$ and $\l_*$ are given, one can simply say that a point $x_* \in \mathcal S_{\l_*}$, regarded as an alias of $q_* = (0,\l_*,x_*)$, is a \emph{bifurcation point} if so is $q_*$.

For a necessary condition and some sufficient conditions for a point $x_*$ of a nontrivial eigensphere to be a bifurcation point see \cite{ChFuPe1}.
Results regarding the existence of (global or non-global) bifurcation points belonging to even-dimensional eigenspheres can be found in \cite{BeCaFuPe-s1, BeCaFuPe-s2, BeCaFuPe-s3, BeCaFuPe-s4, BeCaFuPe-s5, ChFuPe2, ChFuPe3, ChFuPe5}.

\medskip
Theorem \ref{continuation 1} below, which is crucial for our main result (Theorem \ref{ex conjecture}), provides, in particular, a sufficient condition for an isolated subset of trivial solutions of \eqref{perturbed eigenvalue problem} to contain at least one bifurcation point.
To prove it, we need the following lemma of point-set topology, which is particularly suitable to our purposes and is obtained from general results by C.\ Kuratowski (see \cite{Ku}, Chapter $5$, Vol.~$2$).
For an interesting paper on connectivity theory we also recommend \cite{Alex}.

\begin{lemma}[\!\cite{FuPe3}\,]
\label{Whyburn}
Let $\K$ be a compact subset of a locally compact metric space $X$.
If every compact subset of $X$ containing $\K$ has nonempty boundary, then $X \setminus \K$ contains a connected set whose closure in $X$ is non-compact and intersects $\K$.
\end{lemma}

Recall that, according to Notation \ref{slice}, given a subset $D$ of $\R\per\R\per\S$ and $s \in \R$, the symbol $D_s$ stands for the $s$-slice of $D$. Namely,
\[
D_s = \big\{(\l,x) \in \R\per\S: (s,\l,x) \in D\big\}.
\]

\begin{theorem}
\label{continuation 1}
Let $\O$ be an open subset of $\,\R\per\R\per\S$.
If $\deg(\Psi,\O_0,0) \not= 0$, then $\O$ has a connected set of nontrivial solutions whose closure in $\O$ is non-compact and contains at least one bifurcation point.
\end{theorem}
\begin{proof}
Since, by definition, a trivial solution of \eqref{perturbed eigenvalue problem} is a bifurcation point if it is in the closure of the set of nontrivial solutions, the assertion is the same as that of Lemma \ref{Whyburn} provided that $X$ is the set of the solutions in $\O$ and $\K$ is its subset of the trivial ones.
Namely, $X = (\Psi^+)^{-1}(0) \cap \O$ and $\K = \{0\}\per X_0$, where $X_0$ is the $s=0$ slice of $X$, which coincides with $\Psi^{-1}(0)\cap \O_0 = \Sigma_0 \cap \O_0$.

Thus, it is enough to prove that the metric pair $(X,\K)$ satisfies the assumptions of Lemma \ref{Whyburn}.

Let us show first that $X$ is locally compact.
Recall that $\Psi\colon \R\per\S \to H$ is a Fredholm map of index zero.
Therefore, its extension $(s,\l,x) \mapsto \Psi(\l,x)$ is Fredholm of index one, being obtained by composing $\Psi$ with the projection $(s,\l,x) \mapsto (\l,x)$, which is a $\Phi_1$-map (recall the property about the index of the composition of Fredholm operators in Subsection \ref{Algebraic preliminaries}).
Since $\Psi^+$ is obtained by adding to this extension of $\Psi$ the compact $C^1$-map $(s,\l,x) \mapsto sN(x)$, we get that $\Psi^+$ is as well Fredholm of index one.
Therefore, $\Psi^+$, being Fredholm, is a locally proper map (see \cite{Smale}).
This implies that the set $\Sigma = (\Psi^+)^{-1}(0)$ is locally compact, and so is its relatively open subset $X$.
Moreover, $\K = \{0\}\per X_0$ is compact, since $X_0$ coincides with the set $\Psi^{-1}(0)\cap \O_0$, whose compactness is implicit in the assumption that $\deg(\Psi,\O_0,0)$ is defined.

It remains to prove that any compact subset of $X$ containing $\K$ has nonempty boundary in $X$.
Assume, by contradiction, that this is not the case.
Hence there exists a compact subset $D$ of $X$ containing $\K$ whose boundary, in $X$, is empty.
Therefore, $D$ is relatively open in $X$ and, consequently, there exists an open subset $W$ of $\O$ such that $X \cap W = D$.
Incidentally we observe that, according to Definition \ref{isolated set}, the compact set $D$ is isolated among the elements of $X$.

Notice that, since $W \subseteq \O$, one has $D = \big\{(\l,s,x) \in W\colon \Psi^+(\l,s,x) = 0\big\}$. Thus, according to the Generalized Homotopy Invariance Property, $\deg(\Psi^+_s, W_s,0)$ does not depend on $s \in \R$.

Because of the Excision Property, for $s = 0$ one has
\[
\deg(\Psi^+_0, W_0,0) = \deg(\Psi^+_0,\O_0,0).
\]
Therefore, since the partial map $\Psi^+_0$ of $\Psi^+$ coincides with $\Psi$, one gets
\[
\deg(\Psi^+_s, W_s,0) = \deg(\Psi,\O_0,0) \not= 0, \quad \forall s \in \R.
\]
Now, the compactness of $D$ implies that there exists $s_* \in \R$ such that the set $D_{s_*} = (\Psi^+_{s_*})^{-1}(0) \cap W_{s_*}$ is empty.
Consequently, because of the Existence Property one gets $\deg(\Psi^+_{s_*}, W_{s_*},0) = 0$, and the assertion follows from the contradiction.
\end{proof}

Corollary \ref{continuation 2} below provides a sufficient condition for 
the existence of bifurcation of a global branch emanating from a point in an isolated subset of trivial solutions.

In order to deduce it from Theorem \ref{continuation 1}, we need to show that the map $\Psi^+$ is more than locally proper.
Actually,
\begin{itemize}
\item
\emph{$\Psi^+$ is proper on any bounded and closed subset of its domain.}
\end{itemize}
To check this, observe that $\Psi^+$ is the sum of two maps:
one is the restriction $\hat L$ to the manifold $\R\per\R\per\S$ of the linear operator
\[
\bar L\colon \R\per\R\per G \to F, \quad (s,\l,x) \mapsto Lx,
\]
which is Fredholm of index two; the other one is the map $(s,\l,x) \mapsto sN(x) - \l Cx$, which sends bounded sets into relatively compact sets.
The linear operator $\bar L$, being Fredholm, is proper on bounded and closed subsets of its domain.
Therefore, the same property is inherited by $\hat L$ to the closed subset $\R\per\R\per\S$ of $\R\per\R\per G$.
One can check that this property is preserved by adding to $\hat L$ a compact map.

\medskip
From Theorem \ref{continuation 1} we get a sufficient condition for the existence of a global bifurcation point.
Recall that the slice $\Sigma_0$ of the set $\Sigma = (\Psi^+)^{-1}(0)$ of the solutions of \eqref{perturbed eigenvalue problem} coincides with the set $\mathcal S = \Psi^{-1}(0)$ of the eigenpoints of \eqref{eigenvalue problem}.

\begin{corollary}
\label{continuation 2}
Let $\K$ be an isolated subset of $\Sigma_0$ such that $\psideg(\K) \not= 0$.
Then, there exists a connected set of nontrivial solutions of \eqref{perturbed eigenvalue problem} whose closure contains a bifurcation point $q_* \in \{0\}\per \K$ and it is either unbounded or encounters a bifurcation point $q^* \notin \{0\}\per \K$.
Consequently, $q_*$ is a global bifurcation point.
\end{corollary}
\begin{proof}
Let $\O$ be the open subset of $\R\per\R\per\S$ obtained by removing the closed set of the elements of $\{0\}\per\Sigma_0$ which are not in $\{0\}\per \K$ (recall that $\K$, according to Definition \ref{isolated set}, is relatively open in $\Sigma_0$).
Thus, $\O_0$ is an isolating neighborhood of $\K$ among $\Sigma_0$ and, consequently, $\deg(\Psi, \O_0, 0) = \psideg(\K) \not= 0$.
Because of Theorem~\ref{continuation 1}, there exists a connected set $\mathcal C$ of nontrivial solutions whose closure in $\O$, call it $\mathcal C^+$, is non-compact and contains at least one bifurcation point $q_*$, which, necessarily, belongs to $\O$.

It is enough to prove that $\mathcal C$ satisfies the first assertion: the second one is a consequence of the fact that the closure of a connected set is as well connected.
To this purpose, we need to show, first, that $q_* \in \{0\}\per \K$ and, after this, we may assume that $\mathcal C^+$ is bounded.

The point $q_*$ belongs to $\{0\}\per \K$, since, because of the definition of $\O$, one has $\O \cap (\{0\}\per\Sigma_0) = \{0\}\per \K$.

Assume now that ${\mathcal C^+}$ is bounded.
Then, so is the closure $\overline{\mathcal C}$ of ${\mathcal C}$ (in $\R\per\R\per\S$).
It remains to show that $\overline{\mathcal C}$ contains a trivial solution $q^*$ which does not belong to $\{0\}\per \K$.

Recall that $\Psi^+$ is proper on bounded closed subsets of its domain.
Therefore, the subset $\overline{\mathcal C}$ of $(\Psi^+)^{-1}(0)$ is compact.
Moreover, $\overline{\mathcal C}$ contains ${\mathcal C^+}$, which is not compact.
This implies that $\overline{\mathcal C}$ has at least one point $q^*$ which is not in ${\mathcal C^+}$.
The fact that $q^*$ is a bifurcation point not in $\{0\}\per \K$ follows from the definition of $\O$.
\end{proof}

\begin{corollary}
\label{compact component}
If $D$ is a compact component of $\Sigma$, then $\psideg(D_0) = 0$.
\end{corollary}
\begin{proof}
Observe that $D_0$ is an isolated subset of $\Sigma_0$, due to the fact that the set of all the eigenvalues of \eqref{eigenvalue problem} is discrete.

Suppose, by contradiction, that $\psideg(D_0) \not=0$.
Then, Corollary \ref{continuation 2} applies ensuring the existence of a connected set $\mathcal C$ of nontrivial solutions of \eqref{perturbed eigenvalue problem} whose closure $\overline{\mathcal C}$, which is as well connected, contains at least two trivial solutions: one, say $q_*$, belonging to $\{0\}\per D_0$, and one, call it $q^*$, outside $\{0\}\per D_0$.

Since $q_*$ belongs to both the connected set $\overline{\mathcal C}$ and the component $D$, one gets $\overline{\mathcal C} \subseteq D$.
Therefore, $q^*$ belongs to $D$ and, consequently, being trivial, belongs as well to $\{0\}\per D_0$, which is a contradiction yielding the assertion.
\end{proof}

We are ready to prove our main achievement (Theorem \ref{ex conjecture}).
Its proof is based on previous results requiring the notion of degree for the oriented map $\Psi$ and the convenient hypothesis $\dim G > 1$.
In spite of this, its assertion is still valid when the space $G$ has dimension one, as Example \ref{example1} shows.

\begin{theorem}
\label{ex conjecture}
Let $(\l_*,x_*)$ be a simple eigenpoint of problem \eqref{eigenvalue problem}.
Then, in the set $\Sigma$ of the solutions of \eqref{perturbed eigenvalue problem}, the connected component containing $(0,\l_*,x_*)$ is either unbounded or includes a trivial solution $(0,\l^*,x^*)$ with $\l^* \not= \l_*$.
\end{theorem}
\begin{proof}
We may assume that the connected component $D$ of $\Sigma$ containing $(0,\l_*,x_*)$ is bounded, and we need to show that its slice $D_0$ is not contained in $\{\l_*\}\per{\mathcal S_{\l_*}}$.

Recalling that $\Sigma = (\Psi^+)^{-1}(0)$ and that $\Psi^+$ is proper on bounded closed subsets of its domain, we get that $D$ is compact.
Then, Corollary \ref{compact component} implies $\psideg(D_0) = 0$ and, consequently, because of Theorem \ref{psi-degree simple eigenpoint}, $D_0 \not\subseteq \{\l_*\}\per{\mathcal S_{\l_*}}$.
\end{proof}

\section{Some illustrating examples and an application}
\label{Examples and application}

We give now three examples illustrating the assertion of Theorem \ref{ex conjecture}.
The last one shows also that, in this theorem, the assumption that the solution $(0,\l_*,x_*)$ is simple cannot be removed.

After the examples, we will give an application of Theorem \ref{ex conjecture} to a motion equation containing a nonlinearity like an air resistance force.

\subsection{Examples}
\label{examples}
The first example regards an exhaustive discussion of the solutions of problem \eqref{perturbed eigenvalue problem} in the case when $\dim G = 1$.
As we shall see, the assertion of Theorem \ref{ex conjecture} holds true also in this minimal dimension.

\begin{example} 
\label{example1}
Let $G=H=\R$ and consider the problem

\begin{equation}
\label{problem in E1}
\left\{
\begin{array}{ccc}
lx + s N(x) \eql \l cx,\\
|x| \eql 1,
\end{array}\right.
\end{equation}
in which $l$ and $c$ are two given real numbers, and $N$ is an arbitrary real function.

We assume $c \not=0$, so that the unperturbed problem (obtained by putting $s= 0$) has a unique eigenvalue, $\l_* = l/c$, and two corresponding twin eigenpoints:
\[
p = (\l_*,x_*) = (l/c,1), \quad \bar p = (\l_*,-x_*) = (l/c,-1),
\]
both simple.
We will interpret the assertion of Theorem \ref{ex conjecture} in this extreme situation.

For a solution $(s,\l,x)$ of problem \eqref{problem in E1} we have two possibilities: $x = 1$ or $x = -1$.

For $x = 1$ one has $\l = (l+sN(1))/c$.
Thus, the set of solutions of this type is given by the straight line
\[
\Sigma_+ = \big\{\big(s,(l+sN(1))/c,1\big) \in \R^3: s \in \R \big\}.
\]
Analogously, for $x = -1$ one gets
\[
\Sigma_- = \big\{\big(s,(l-sN(-1))/c,-1\big) \in \R^3: s \in \R \big\}.
\]
Therefore, the set $\Sigma$ of all the solutions of \eqref{problem in E1} is $\Sigma_+ \cup \Sigma_-$, and the assertion of Theorem \ref{ex conjecture} is satisfied for both the simple eigenpoints $(\l_*,x_*)$ and $(\l_*,-x_*)$.
\end{example}

The following example concerns a differential equation with an evident physical meaning, and the parameter $2s$, when positive, can be regarded as a frictional coefficient.
Its abstract formulation has infinitely many eigenpoints, all of them simple.
The set $\Sigma$ of the solution triples $(s,\l,x)$ is the union of infinitely many unbounded components, each of them corresponding to one and only one eigenpoint.

\begin{example} 
\label{example2}
Let us show how Theorem \ref{ex conjecture} agrees with the structure of the non-zero solutions of the following boundary value problem:
\begin{equation}
\label{problem in E2}
\left\{
\begin{array}{lccc}
x''(t) + 2s x'(t) + \l x(t) = 0,\\
x(0) = 0 = x(\pi).
\end{array}\right.
\end{equation}
For this purpose, we will interpret it as an abstract problem of the type \eqref{ex conjecture} by specifying what are here the spaces $G$ and $H$, the linear operators $L$ and $C$, and the map $N$.

Let $H^1(0,\pi)$ be the Hilbert space of the absolutely continuous real functions defined in $[0,\pi]$ with derivative in $L^2(0,\pi)$, and denote by $H^2(0,\pi)$ the Hilbert space of the $C^1$ real functions in $[0,\pi]$ with derivative in $H^1(0,\pi)$.

Clearly $H^1(0,\pi)$ is a subset of the Banach space $C[0,\pi]$, and the inclusion, we recall, is a compact operator.
Therefore the injection of $H^1(0,\pi)$ into $L^2(0,\pi)$ is as well compact, due to the bounded injection of $C[0,\pi]$ into $L^2(0,\pi)$.
Analogously, the inclusion of $H^2(0,\pi)$ into $C^1[0,\pi]$ is compact and the inclusion of $C^1[0,\pi]$ into $H^1[0,\pi]$ is continuous.

As a source space $G$ we take the $2$-codimensional closed subspace of $H^2(0,\pi)$ consisting of the functions $x$ satisfying the boundary condition $x(0) = 0 = x(\pi)$.
The target space $H$ is $L^2(0,\pi)$.

Observe that the second derivative $x \mapsto x''$, as a linear operator from $H^2(0,\pi)$ to $L^2(0,\pi)$, is bounded and Fredholm of index $2$, being surjective with $2$-dimensional kernel.
Therefore, its restriction $L \in \L(G,H)$ is a $\Phi_0$-operator (recall the property about the composition of Fredholm operators).

Here $C$ associates to any $x \in G$ the element $-x \in H$.
Thus, $C$ is a compact linear operator, since so is the injection of $H^2(0,\pi)$ into $L^2(0,\pi)$.
The map $N$ transforms $x \in G$ in $2x' \in H$ and, therefore, it is as well compact, as composition of a bounded linear operator into $H^1(0,\pi)$ with the compact injection into $L^2(0,\pi)$.

Among the infinitely many equivalent norms in $H^2(0,\pi)$ and, consequently, in $G$ we choose the one associated with the inner product
\[
\langle x,y \rangle = \frac{1}{\pi}\int_0^\pi\pb\big(x(t)y(t) + x''(t)y''(t)\big)\, dt.
\]

As in the previous sections, $\S$ denotes the unit sphere of $G$.
Since we are interested in the non-zero solutions of \eqref{problem in E2}, the linearity of $N$ justifies the condition $x \in \S$ in the following abstract formulation of our problem:
\begin{equation}
\label{abstract problem in E2}
\left\{
\begin{aligned}
&Lx + s N(x) = \l Cx\\
&x \in \S.
\end{aligned}\right.
\end{equation}

Elementary computations show that the eigenvalues of the unperturbed problem (obtained with $s=0$) are $\l_1 = 1, \l_2 = 4,\dots,\l_n=n^2,\dots$ and to any $\l_n$ corresponds the $1$-dimensional eigenspace $\R x_n = \Ker(L-\l_n C)$, with $x_n \in \S$ given~by
\[
x_n(t) = \sqrt{\frac{2}{1+n^4}}\sin(nt).
\]

Let us show that $p_n = (\l_n,x_n)$ and $\bar p_n = (\l_n,-x_n)$ are simple eigenpoints, according to Definition~\ref{simple}.
Since $C$ is compact, the operator $T_n = L - \l_n C$ is Fredholm of index zero.
Therefore, we need only to prove that $Cx_n$ does not belong to $T_n(G)$, which means that the equation $T_n(x) = Cx_n$ has no solutions in $G$.
In fact, there are no solutions of the resonant problem
\begin{equation*}
\label{no solution in E2}
\left\{
\begin{array}{lccc}
x''(t) + n^2 x(t) = \sin(nt),\\
x(0) = 0 = x(\pi).
\end{array}\right.
\end{equation*}

With standard computations one can prove that, given $s \in \R$, the differential equation $x''(t) + 2s x'(t) + \l x(t) = 0$ has a non-zero solution verifying the boundary condition $x(0) = 0 = x(\pi)$ if and only if $\l = n^2+s^2$, with $n \in \mathbb N$.
Therefore, the subset $\mathcal E$ of the $s\l$-plane of the eigenpairs of \eqref{abstract problem in E2} is composed of the disjoint union of infinitely many parabolas of equation $\l = n^2 + s^2$, $n \in \mathbb N$.
Moreover, given $(s,n^2+s^2) \in \mathcal E$, any solution in $G$ of the differential equation
\begin{equation*}
\label{problem in E2 with eigenpair}
x''(t) + 2s x'(t) + (s^2+n^2)x(t) = 0
\end{equation*}
belongs to the straight line $\R x_{s,n}$, where $x_{s,n}(t) = \exp(-st)\sin(nt)$.

As a consequence of this, given any eigenpoint $p_n = (\l_n, x_n)$, the connected component, in $\Sigma$, containing the corresponding trivial solution $(0,\l_n, x_n)$ is the unbounded curve
\[
\big\{(s,s^2+n^2,x_{s,n}/\|x_{s,n}\|): s \in \R \big\}.
\]
Obviously, for the twin eigenpoint $\bar p_n = (\l_n,-x_n)$, one gets
\[
\big\{(s,s^2+n^2,-x_{s,n}/\|x_{s,n}\|): s \in \R \big\}.
\]
In conclusion, for any eigenpoint the assertion of Theorem \ref{ex conjecture} is verified.
\end{example}

\medskip
Example \ref{example3} below has already been considered in \cite{BeCaFuPe-s4} in relation to the conjecture formulated there and solved by Theorem \ref{ex conjecture} above.
It concerns a system of two ordinary differential equations with periodic boundary conditions, and the set $\Sigma$ of its solutions $(s,\l,x)$ has a component which is diffeomorphic to a circle and contains exactly four trivial solutions, all of them simple.
These four solutions are associated with two eigenvalues of the unperturbed problem: a pair of twins for each eigenvalue.
The other components of $\Sigma$ are infinitely many $1$-dimensional spheres (geometric circles).
The projection of each of them into the $s\l$-plane is a singleton $\{(0,\l)\}$, with $\l$ an eigenvalue of the unperturbed problem.

\begin{example}
\label{example3}
We are interested in the non-zero solutions of the following system of coupled differential equations with $2\pi$-periodic boundary conditions:
\begin{equation}
\label{system 3 in E3}
\begin{cases}
x'_1(t) + x_1(t) - s x_1(t) = \l x_2(t),\\
x'_2(t) - x_2(t) - s x_2(t) = -\l x_1(t),\\
x_1(0)=x_1(2\pi), \; x_2(0)= x_2(2\pi).
\end{cases}
\end{equation}
As in Example \ref{example2}, we interpret our problem in the abstract form
\begin{equation}
\label{abstract problem in E3}
\left\{
\begin{aligned}
&Lx + s N(x) = \l Cx\\
&x \in \S,
\end{aligned}\right.
\end{equation}
where $L$, $C$ and $N$ are operators to be defined below, together with the source and the target spaces $G$ and $H$.

Let $H^1((0,2\pi),\R^2)$ be the Hilbert space of the absolutely continuous functions
$
x=(x_1,x_2)\colon [0,2\pi] \to \R^2
$
with derivative in $L^2((0,2\pi),\R^2)$.

We take as $G$ the closed subspace of $H^1((0,2\pi),\R^2)$ of the functions satisfying the periodic condition $x(0)=x(2\pi)$, and as $H$ the space $L^2((0,2\pi),\R^2)$.
Observe that $G$ has codimension $2$ in $H^1((0,2\pi),\R^2)$.
Therefore, the operator $L \colon G \to H$, given by $(x_1,x_2) \mapsto (x_1'+x_1,x_2'-x_2)$, is Fredholm of index zero.

The operators $N$ and $C$, given by $(x_1,x_2) \mapsto (-x_1,-x_2)$ and $(x_1,x_2) \mapsto (x_2,-x_1)$ respectively, are compact, since so is the injection
\[
H^1((0,2\pi),\R^2) \hookrightarrow L^2((0,2\pi),\R^2).
\]

The norm in $G$ is the one associated with the inner product
\[
\langle x,y \rangle_1 = \frac{1}{2\pi}\int_0\sp{2\pi}\big(x(t)\cdot y(t) + x'(t)\cdot y'(t)\big)\, dt,
\]
where, given two vectors $a=(a_1,a_2)$ and $b=(b_1,b_2)$ in $\R^2$, $a \cdot b$ denotes the standard dot product.
As in the previous sections, $\S$ is the unit sphere in $G$.

The eigenvalues of the unperturbed problem (obtained by putting $s=0$ in \eqref{abstract problem in E3}) are $\l = \pm \sqrt{1+n^2},\; n = 0,1,2,\dots$ and the set $\mathcal E$ of the eigenpairs of \eqref{abstract problem in E3} is the disjoint union of two sets: the circle
\[
\mathcal C = \big\{(s,\l) \in \R^2: s^2+\l^2=1\big\}
\]
and the isolated eigenpairs
\[
\big\{(0,\l) \in \R^2: \l = \pm \sqrt{1+n^2},\; n = 1,2,3,\dots \big\}.
\]
Notice in fact that, if $(s,\l,x)$ satisfies \eqref{system 3 in E3} with $s\neq 0$, 
then $x$ must be constant.
Thus, the non-zero solutions of the periodic boundary value problem \eqref{system 3 in E3}
are of two types: the constant ones, corresponding to the eigenpairs of the circle $\mathcal C$, and the oscillating ones associated with the isolated eigenpairs.

\medskip
Let us examine first the case $(s,\l) \in \mathcal C$.
One has $(s,\l)=(\cos\t, \sin\t)$, with $\t \in [0,2\pi]$, and the kernel of the linear operator
\[
L + (\cos\t)N -(\sin\t)C \in \L(G,H)
\]
is the straight line $\R x^\t$, where $x^\t \in \S$ is the constant function
\[
x^\t \colon [0,2\pi] \to \R^2, \quad t \mapsto (\cos(\t/2),\sin(\t/2)).
\]
Therefore, the connected component $D$ of $\Sigma$ containing the trivial solution
\[
q_* = (0,\l_*,x_*) =(0,1,x^{\pi/2})
\]
of \eqref{abstract problem in E3} is diffeomorphic to a circle, as its parametrization
\[
\t \in [0,4\pi] \mapsto (\cos\t,\sin\t,x^\t) \in \R\times\R\times\S
\]
shows.
This component contains four trivial solutions: two of them associated with the eigenvalue $\l_*=1$ and the others with $\l^*=-1$.
They are all simple solutions, and the component $D$ agrees with the statement of Theorem \ref{ex conjecture}.

Incidentally, we observe that the projection of $D$ onto the circle $\mathcal C$ is a double covering map, and the above parametrization of $D$ is just a lifting of the map $\t \mapsto (\cos\t,\sin\t) \in \mathcal C$, $\t \in [0,4\pi]$.

\smallskip
Let us consider now the isolated eigenpairs.
That is, the ones having $|\l| > 1$.
In this case \eqref{system 3 in E3} admits non-zero solutions if and only if $s=0$ and $\l = \pm\sqrt{1+n^2}$, with $n = 1,2,3,\dots$
More precisely, these solutions are oscillating and, given any isolated eigenpair $(0,\l_*)$, the corresponding solutions of \eqref{system 3 in E3}, plus the zero one, form a two-dimensional subspace of $H^1((0,2\pi),\R^2)$.
This implies that the eigensphere $\mathcal S_{\l_*}$ of the unperturbed problem is the geometric circle $\Ker(L-\l_* C) \cap \S$.
Therefore, if $x_*$ is any element of this circle, the connected component in $\Sigma$ containing the corresponding trivial solution $(0,\l_*,x_*)$ does not satisfy the assertion of Theorem \ref{ex conjecture}.
Thus, the assumption that the eigenpoint $(\l_*,x_*)$ is simple cannot be removed.
\end{example} 

\subsection{An application}
\label{an application}
We close by showing how both Theorem \ref{ex conjecture} and the well-known notion of winding number allow us to deduce theoretically, without explicitly solving the differential equation, that the structure of set $\Sigma$ of solutions $(s,\l,x)$ of the nonlinear boundary value problem
\begin{equation*}
\label{air resistance}
\left\{
\begin{array}{lccc}
x''(t) + s g(x'(t)) + \l x(t) = 0,\\
x(0) = 0 = x(\pi),\; x \in \S
\end{array}\right.
\end{equation*}
is essentially the same as in Example \ref{example2}.
Here $g\colon \R \to \R$ is an increasing odd $C^1$-function, as is the classical air resistance force $g(v) = v|v|$, and the sphere $\S$ is as in Example \ref{example2}.
The parameter $s$, when positive, may be regarded as a frictional coefficient.

The problem can be rewritten in the abstract form as follows:
\begin{equation}
\label{abstract problem in A1}
\left\{
\begin{aligned}
&Lx + s N(x) = \l Cx\\
&x \in \S.
\end{aligned}\right.
\end{equation}
The spaces $G$ and $H$ are the same as in Example \ref{example2}, and so are the operators $L$ and $C$.
The compact map $N\colon G \to H$ sends $x$ into the function $N(x)\colon t \mapsto g(x'(t))$.
It is not difficult to prove that $N$ is $C^1$ and its Fr\'echet differential at $x \in G$ is given by $dN_x(h)\colon t \mapsto g'(x'(t))h'(t)$.

The unperturbed problem is the same as in Example \ref{example2}.
Therefore, its eigenvalues are $\l_1 = 1, \l_2 = 4,\dots,\l_n=n^2,\dots$
They are all simple and, consequently, each of them corresponds to a pair of isolated unit eigenpoints.

We will prove that the set $\Sigma$ of the solutions $(s,\l,x)$ of \eqref{abstract problem in A1} contains infinitely many unbounded components, each of them corresponding to one and only one eigenpoint.

\medskip
Let $S^1$ denote the unit circle of $\C$ and let $\mathrm{w}\colon C(S^1) \to \Z$ stand for the winding number function, defined on the set of the continuous maps from $S^1$ into itself.
Recall that, given $\g \in C(S^1)$, $\mathrm w(\g)$ is the same as the Brouwer degree of $\g$ and, speaking loosely, denotes the number of times that $\g$ travels counterclockwise around the origin of $\C$, and it is negative if the curve travels clockwise.

Call $\wj$ the integer valued function that to any non-zero solution $x$ of the parametrized differential equation
\begin{equation}
\label{differential equation}
x''(t) + s g(x'(t)) + \l x(t) = 0,
\end{equation}
depending on $s,\l \in \R$, assigns the winding number $\wj(x)$ of the closed curve $\j(x)\in C(S^1)$ defined by
\[
z = e^{i\t} \mapsto \frac{\big(x'(\t/2) + i x'(0)x(\t/2)\big)^2}{x'(\t/2)^2+x'(0)^2x(\t/2)^2}, \quad \t \in [0,2\pi].
\]
Notice that, given any non-zero solution $x$ of \eqref{differential equation}, $\j(x)$ is well defined, since $x(t)$ and $x'(t)$ cannot be simultaneously zero, due to the uniqueness of the Cauchy problem.
Observe also that $\j(x)$ is a closed curve, since both the endpoints coincide with $1 \in \C$.

It is convenient to extend the map $x \mapsto \j(x)$ to the symmetric set of all the functions $x \in G$ having the property that $x(t)^2+x'(t)^2 > 0$ for all $t \in [0,\pi]$.
We denote this set by $\X$ and we observe that it is open, because of the bounded inclusions $H^2(0,\pi) \hookrightarrow C^1[0,\pi]$ and $H^1(0,\pi) \hookrightarrow C[0,\pi]$.

One can check that, for example, if $x(t) = \sin(nt)$ with $n \in \Z$, then $\j(x)$ is the map $z \mapsto z^n$, whose winding number is $n$.

One can also check that, if $a$ is a positive constant and $x \in \X$, then $\j(x)$ and $\j(ax)$ are homotopic, therefore they have the same winding number.
Moreover, if $x \in \X$, then $\j(-x) = 1/\j(x)$, which is the same as the conjugate map $\overline{\j(x)}$ of $\j(x)$.
Therefore, the winding numbers of $\j(x)$ and $\j(-x)$ are opposite each other, and this happens for the two unit eigenvectors corresponding to any eigenvalue $\l_n = n^2$ of our problem.
In fact, this number is $n$ for the unit eigenfunction
\[
x_n(t) = \sqrt{\frac{2}{1+n^4}}\sin(nt),
\]
and $-n$ for the opposite one.

Observe that, due to the fact that $\X$ is open in $G$, if two functions of $\X$ are sufficiently close, then the segment joining them lies in $\X$.
Therefore, the corresponding two images under $\j\colon \X \to C(S^1)$ are homotopic and, consequently, they have the same winding number.
Thus, the integer valued function $\wj\colon \X \to \Z$ is locally constant.

Since the projection map $\mathrm p\colon \Sigma \to \X$ that to any solution $(s,\l,x)$ of \eqref{abstract problem in A1} assigns the function $x$ is continuous, we have the following

\begin{remark}
\label{locally constant}
The map $\mathrm{wjp}\colon \Sigma \to \Z$ that to any solution $(s,\l,x)$ of \eqref{abstract problem in A1} assigns the winding number of the closed curve $\j(x)$ is locally constant.
\end{remark}

Let $q_* = (0,\l_*,x_*)$ be any trivial solution of \eqref{abstract problem in A1}.
We want to prove that the connected component $D_*$ of $\Sigma$ containing $q_*$ is unbounded and does not meet other trivial solutions.

To this purpose, observe first that, since $D_*$ is connected, Remark \ref{locally constant} implies $\mathrm{wjp}(q) = \mathrm{wjp}(q_*)$ for all $q \in D_*$.
In particular, if $\l_* = \l_n = n^2$, then $\mathrm{wjp}(q_*)$ is $n$ or $-n$, depending on whether $x_*$ is the above function $x_n$ or its opposite.
Thus, $D_*$ does not contain trivial solutions different from $q_*$, consequence of the fact that the function that to any trivial solution $q$ assigns the integer $\mathrm{wjp}(q)$ is injective.

Finally, from Theorem \ref{ex conjecture} we get that $D_*$ is unbounded, since otherwise $D_*$ would contain a trivial solution different from $q_*$.



\end{document}